\documentclass[12pt,a4paper,reqno]{amsart} 
\usepackage{amssymb}
\usepackage{latexsym}
\usepackage{amsmath}
\usepackage{mathrsfs}
\usepackage[X2,T1]{fontenc}
\usepackage{cite}
\usepackage{calc}                   

\newcommand{\scal}[2]{\langle #1,#2\rangle}
\newcommand{\rr}[1]{\mathbf R^{#1}}
\newcommand{\rrstar}[1]{\mathbf R^{#1}_*}
\newcommand{\zz}[1]{\mathbf Z^{#1}}

\newcommand{\nm}[2]{\Vert #1\Vert _{#2}}

\newcommand{\op}{\operatorname{Op}}

\newcommand{\sets}[2]{\{ \, #1\, ;\, #2\, \} }
\newcommand{\ep}{\varepsilon}

\newcommand{\cdo}{\, \cdot \, }
\newcommand{\supp}{\operatorname{supp}}

\newcommand{\eabs}[1]{\langle #1\rangle}     

\newcommand{\vrum}{\vspace{0.1cm}}

\newcommand{\maclD}{\mathcal D}
\newcommand{\maclE}{\mathcal E}
\newcommand{\maclS}{\mathcal S}

\newcommand{\mascE}{\mathscr E}
\newcommand{\mascF}{\mathscr F}

\newcommand{\mascP}{\mathscr P}

\newcommand{\mabfp}{\mathbf p}
\newcommand{\mabfq}{\mathbf q}
\newcommand{\mabfr}{\mathbf r}

\newcommand{\sfW}{\mathsf{W}}

\newcommand{\Co}{\mathsf{Co}}

\numberwithin{equation}{section}          

\newtheorem{thm}{Theorem}
\numberwithin{thm}{section}

\newcommand{\rubrik}{}
\newtheorem{prop}[thm]{Proposition}
\newtheorem{cor}[thm]{Corollary}
\newtheorem{lemma}[thm]{Lemma}

\theoremstyle{definition}

\newtheorem{defn}[thm]{Definition}

\theoremstyle{remark}
\newtheorem{rem}[thm]{Remark}              

\title{\textbf {Gabor analysis for a broad class of quasi-Banach
modulation spaces}}

\author{Joachim Toft}

\address{Department of Mathematics,
Linn{\ae}us University, V{\"a}xj{\"o}, Sweden}

\email{joachim.toft@lnu.se}

\begin{document}

\begin{abstract}
We extend the Gabor analysis in \cite{GaSa} to a broad class
of modulation spaces, allowing more general mixed quasi-norm
estimates and weights in the definition of the modulation space
quasi-norm. For such spaces we deduce invariance and
embedding properties, and that the elements admit reconstructible
sequence space representations using Gabor frames.
\end{abstract}

\maketitle

\section{Introduction}\label{sec0}

\par

A modulation space is, roughly speaking, a set
of distributions or ultra-distributions, obtained by imposing a
suitable quasi-norm estimate on the short-time
Fourier transforms of the involved distributions.
(See Sections \ref{sec1} and \ref{sec2} for definitions.)

\par

In \cite{GaSa}, Galperin and Samarah establish fundamental continuity
and invariance properties for modulation spaces of the form
$M^{p,q}_{(\omega )}$, when $\omega$ is a polynomially moderate weight
and $p,q\in (0,\infty ]$. More precisely, Galpering and Samarah
prove in \cite{GaSa} among others that the following fundamental
properties for such modulation spaces hold true:
\begin{itemize}
\item[(1)] $M^{p,q}_{(\omega )}$ is independent of the choice of involved window 
function in the short-time Fourier transforms;

\vrum

\item[(2)] $M^{p,q}_{(\omega )}$ increases with respect to the parameters
$p$ and $q$, and decreases with respect to $\omega$;

\vrum

\item[(3)] $M^{p,q}_{(\omega )}$ admit reconstructible sequence space
representations using Gabor frames.
\end{itemize}

\par

Note that in contrast to what is usual in modulation spaces theory, the
Lebesgue exponents $p$ and $q$ above are allowed to be
strictly smaller than $1$. This leads to a more comprehensive and difficult
analysis of $M^{p,q}_{(\omega )}$ when $p$ and $q$ are allowed to
stay in $(0,\infty ]$, compared to what is needed when $p$ and $q$
stays in the smaller interval $[1,\infty ]$. In fact, the theory of classical
modulation spaces was established and further developed in
\cite{Fei1,Fei3,Fei5,FG1,FG2,Gc2} by
Feichtinger and Gr{\"o}chenig. In these investigations, Feichtinger and
Gr{\"o}chenig only considered modulation spaces $M^{p,q}_{(\omega )}$
with $p,q\in [1,\infty ]$, and the analysis for deducing the properties
(1)--(3) above is less comprehensive and less difficult compared to the
analysis in \cite{GaSa}.

\par

We also remark that the results in \cite{GaSa} have impact
on unifications of the modulation space theories in
\cite{BorNiel1,BorNiel2,WaHu,Toft8,Toft11}. In fact, in
\cite{BorNiel1,BorNiel2,WaHu,Toft8,Toft11}, certain restrictions are imposed on
the window functions in the definitions of the modulation space quasi-norms.
Due to (1) above, $M^{p,q}_{(\omega )}$ is equal to the corresponding spaces in
\cite{BorNiel1,BorNiel2,WaHu,Toft8,Toft11},
provided the weight $\omega$ and the exponents $p$ and $q$ agree with
those in \cite{BorNiel1,BorNiel2,WaHu,Toft8,Toft11}.

\medspace

The aim of the paper is to the deduce general properties for a broad family of
modulation spaces, which contains the modulation spaces $M^{p,q}_{(\omega )}$
when $p,q\in (0,\infty ]$ and $\omega$ is an \emph{arbitrary} moderate weight. In
particular, the assumption that $\omega$ should be polynomially moderate
is relaxed.
More precisely, we use the framework in \cite{GaSa} and show that (1)--(3) above
still holds for this extended family of modulation spaces. If
the weights are not polynomially moderate, then the involved modulation
spaces do not stay between the Schwartz space $\mathscr S$ and its
dual space $\mathscr S'$. In this situation, the theory is formulated in
the framework of the Gelfand-Shilov space $\Sigma _1$ and its dual
space $\Sigma _1'$ of Gelfand-Shilov ultra-distributions.
Furthermore we allow more general mixed quasi-norm estimates on the
short-time Fourier transform, in the definitions of modulation space
quasi-norms. (See Proposition \ref{AnSynthOpProp} and
Theorem \ref{ModGabThm}.)

\par

In the end of Section \ref{sec3} we use these results to establish identification
properties for compactly supported elements in modulation and
Fourier Lebesgue spaces. In particular we extend the assertions in Remark
4.6 in \cite{RSTT} to more general weights and Lebesgue exponents.
(See Proposition \ref{CompSuppInv}.)

\par

The \emph{classical} modulation spaces $M^{p,q}_{(\omega )}$,
$p,q \in [1,\infty]$ and $\omega$ polynomially moderate
weight on the phase (or time-frequency shift) space, were
introduced by Feichtinger in \cite{Fei1}.
From the definition it follows that $\omega$, and to some
extent the parameters $p$ and $q$
quantify the degrees of asymptotic decay and singularity of
the distributions in $M^{p,q}_{(\omega )}$. The theory of modulation
spaces was developed further and generalized in several ways, e.{\,}g.
in \cite{Fei3,Fei5,FG1,FG2,FGW,Gc1,Gc2}, where
among others, Feichtinger and Gr{\"o}chenig established the theory of
coorbit spaces.

\par

From the construction of modulation spaces
spaces, it turns out that these spaces and Besov spaces in some
sense are rather similar, and sharp embeddings can be found
in \cite{Toft2}, which
are improvements of certain embeddings in \cite {Grobner}. (See
also \cite {SugTom1,WaHu} for verification of the sharpness, and
\cite{Grobner,ToWa,HaWa} for further generalizations in terms of
$\alpha$-modulation spaces.)

\par

During the last 15 years many results appeared which confirm
the usefulness of the modulation spaces. For
example, in \cite{FG1,Gc1,Gc2}, it
is shown that all modulation spaces admit
reconstructible sequence space representations using Gabor frames.
Important reasons for such links are that
$M^{p,q}_{(\omega )}$ may in straight-forward ways be considered
within the coorbit space theory.

\par

More broad families of modulation spaces have been considered since
\cite{Fei1}. For example, in \cite{Fei5}, Feichtinger considers general classes
of modulation spaces, defined by replacing the $L^{p,q}_{(\omega )}$
norm estimates of the short-time Fourier transforms, by more general norm estimates.
Furthermore, in \cite{PT1,PT2,Toft8,Toft11}, the conditions on involved
weight functions are relaxed, and modulation spaces are considered in
the framework of the theory of Gelfand-Shilov distributions. In this setting,
the family of modulation spaces are broad compared to \cite{Fei1,GaSa}.
For example, in contrast to \cite{Fei1,GaSa}, we may have $\mathscr
S' \subseteq M^{p,q}_{(\omega )}$, or $M^{p,q}_{(\omega )}\subseteq
\mathscr S$, for some choices of $\omega$ in \cite{PT1,PT2,Toft8,Toft11}.
Some steps in this direction can be found already in e.{\,}g. \cite{Gc1,Gc2}.

\par

Finally we remarks that in \cite{Rau1,Rau2}, Rauhut extends essential parts of
the coorbit space theory in \cite{FG1,Gc1} to the
case of quasi-Banach spaces. Here it is also shown that
modulation spaces of quasi-Banach types in \cite{GaSa} fit well
in this theory, and we remark that the results in Sections \ref{sec2} and
\ref{sec3} show that our extended family of modulation spaces also meets the
coorbit space theory in \cite{Rau2} well.

\par

\section{Preliminaries}\label{sec1}

\par

In this section we explain some results available in the literature,
which are needed later on, or clarify the subject. The proofs are
in general omitted. Especially we recall some facts about weight
functions, Gelfand-Shilov spaces, and modulation spaces.

\par

\subsection{Weight functions}

\par

We start by discussing general properties on the involved weight
functions. A \emph{weight} on $\rr d$ is a positive function $\omega
\in  L^\infty _{loc}(\rr d)$, and for each compact set $K\subseteq
\rr d$, there is a constant $c>0$ such that
$$
\omega (x)\ge c\qquad \text{when}\qquad x\in K.
$$
A usual condition on $\omega$ is that it should be \emph{moderate},
or \emph{$v$-moderate} for some positive function $v \in
 L^\infty _{loc}(\rr d)$. This means that
\begin{equation}\label{moderate}
\omega (x+y) \le C\omega (x)v(y),\qquad x,y\in \rr d.
\end{equation}
for some constant $C$ which is independent of $x,y\in \rr d$.
We note that \eqref{moderate} implies that $\omega$ fulfills
the estimates
\begin{equation}\label{moderateconseq}
C^{-1}v(-x)^{-1}\le \omega (x)\le Cv(x),\quad x\in \rr d.
\end{equation}
We let $\mascP _E(\rr d)$ be the set of all moderate weights on $\rr d$.
Furthermore, if $v$ in \eqref{moderate} can be chosen as a polynomial,
then $\omega$ is called \emph{polynomially moderate}, or a weight of
\emph{polynomial type}. We let
$\mascP (\rr d)$ be the set of all weights of polynomial type.

\par

It can be proved that if $\omega \in \mascP _E(\rr d)$, then
$\omega$ is $v$-moderate for some $v(x) = e^{r|x|}$, provided the
positive constant $r$ is large enough. In particular,
\eqref{moderateconseq} shows that for any $\omega \in \mascP
_E(\rr d)$, there is a constant $r>0$ such that
$$
e^{-r|x|}\lesssim \omega (x)\lesssim e^{r|x|},\quad x\in \rr d.
$$
Here $A\lesssim B$ means that $A\le cB$
for a suitable constant $c>0$. 

\par

We say that $v$ is
\emph{submultiplicative} if $v$ is even and \eqref{moderate}
holds with $\omega =v$. In the sequel, $v$ and $v_j$ for
$j\ge 0$, always stand for submultiplicative weights if
nothing else is stated.

\par

\subsection{Gelfand-Shilov spaces}

\par

Next we recall the definition of Gelfand-Shilov spaces.

\par

Let $0<h,s,t\in \mathbf R$ be fixed. Then we let $\mathcal S_{t,h}^s(\rr d)$
be the set of all $f\in C^\infty (\rr d)$ such that
\begin{equation*}
\nm f{\mathcal S_{t,h}^s}\equiv \sup \frac {|x^\beta \partial ^\alpha
f(x)|}{h^{|\alpha | + |\beta |}\alpha !^s\, \beta !^t}
\end{equation*}
is finite. Here the supremum should be taken over all $\alpha ,\beta \in
\mathbf N^d$ and $x\in \rr d$. For conveniency we set
$\mathcal S_{s,h}=\mathcal S_{s,h}^s$.

\par

Obviously $\mathcal S_{t,h}^s\subseteq
\mathscr S$ is a Banach space which increases with $h$, $s$ and $t$.
Furthermore, if $s,t>1/2$, or $s,t =1/2$ and $h$ is sufficiently large, then $\mathcal
S_{t,h}^s$ contains all finite linear combinations of Hermite functions.
Since such linear combinations are dense in $\mathscr S$, it follows
that the dual $(\mathcal S_{t,h}^s)'(\rr d)$ of $\mathcal S_{t,h}^s(\rr d)$ is
a Banach space which contains $\mathscr S'(\rr d)$.

\par

The \emph{Gelfand-Shilov spaces} $\mathcal S_{t}^s(\rr d)$ and
$\Sigma _{t}^s(\rr d)$ are the inductive and projective limits respectively
of $\mathcal S_{t,h}^s(\rr d)$ with respect to $h$. This implies that
\begin{equation}\label{GSspacecond1}
\mathcal S_{t}^s(\rr d) = \bigcup _{h>0}\mathcal S_{t,h}^s(\rr d)
\quad \text{and}\quad \Sigma _{t}^s(\rr d) =\bigcap _{h>0}\mathcal
S_{t,h}^s(\rr d),
\end{equation}
and that the topology for $\mathcal S_{t}^s(\rr d)$ is the
strongest possible one such that each inclusion map
from $\mathcal S_{t,h}^s(\rr d)$ to $\mathcal S_{t}^s(\rr d)$
is continuous. The space $\Sigma _t^s(\rr d)$ is a Fr{\'e}chet
space with semi norms $\nm \cdo{\mathcal S_{s,h}^t}$, $h>0$.
Moreover, $\mathcal S _t^s(\rr d)\neq \{ 0\}$, if and only if
$s,t>0$ satisfy $s+t\ge 1$, and $\Sigma _t^s(\rr d)\neq \{ 0\}$,
if and only if satisfy $s+t\ge 1$ and $(s,t)\neq (1/2,1/2)$.

\par

For convenience we set $\mathcal S_{s}=\mathcal S_{s}^s$ and
$\Sigma _{s}=\Sigma _{s}^s$, and remark that $\mathcal S_s(\rr d)$
is zero when $s<1/2$, and that $\Sigma _s(\rr d)$ is zero
when $s\le 1/2$. For each $\ep >0$ and $s,t>0$ such that $s+t\ge 1$,
we have
$$
\Sigma _t^s (\rr d)\subseteq \mathcal S_t^s(\rr d)\subseteq
\Sigma _{t+\ep}^{s+\ep}(\rr d).
$$
On the other hand, in \cite{Pil} there is an alternative elegant definition of
$\Sigma _{s_1}(\rr d)$ and $\mathcal S _{s_2}(\rr d)$ such that these spaces
agrees with the definitions above when $s_1>1/2$ and $s_2\ge 1/2$, but
$\Sigma _{1/2}(\rr d)$ is non-trivial and contained in $\mathcal S_{1/2}(\rr d)$.

\par

From now on we assume that $s,t>1/2$ when considering
$\Sigma _t^s(\rr d)$.

\par

\medspace

The \emph{Gelfand-Shilov distribution spaces} $(\mathcal S_{t}^s)'(\rr d)$
and $(\Sigma _{t}^s)'(\rr d)$ are the projective and inductive limit
respectively of $(\mathcal S_{t,h}^s)'(\rr d)$.  This means that
\begin{equation}\tag*{(\ref{GSspacecond1})$'$}
(\mathcal S_{t}^s)'(\rr d) = \bigcap _{h>0}(\mathcal
S_{t,h}^s)'(\rr d)\quad \text{and}\quad (\Sigma _{t}^s)'(\rr d)
=\bigcup _{h>0} (\mathcal S_{t,h}^s)'(\rr d).
\end{equation}
We remark that already in \cite{GS} it is proved that
$(\mathcal S_{t}^s)'(\rr d)$ is the dual of $\mathcal
S_{t}^s(\rr d)$, and if $s>1/2$, then $(\Sigma _{t}^s)'
(\rr d)$ is the dual of $\Sigma _{t}^s(\rr d)$ (also in
topological sense).

\par

The Gelfand-Shilov spaces are invariant or posses convenient
mapping properties under several basic
transformations. For example they are invariant under
translations, dilations, tensor product, and to some extent under
Fourier transformation. Here tensor products of elements in
Gelfand-Shilov distribution spaces are defined
in similar ways as for tensor products for distributions (cf. Chapter
V in \cite{Ho1}). If $s,s_0,t,t_0>0$ satisfy
$$
s_0+t_0\ge 1,\quad s\ge s_0\quad \text{and},\quad t\ge t_0,
$$
and $f,g\in (\mathcal S_{t_0}^{s_0})'(\rr d)\setminus 0$ and 
then $f\otimes g \in (\mathcal S_t^{s})'(\rr {2d})$, if and
only if $f,g\in (\mathcal S_t^{s})'(\rr d)$. Similar facts hold for any
other choice of Gelfand-Shilov spaces of functions or distributions.

\par

From now on we let $\mathscr F$ be the Fourier transform which
takes the form
$$
(\mathscr Ff)(\xi )= \widehat f(\xi ) \equiv (2\pi )^{-d/2}\int _{\rr
{d}} f(x)e^{-i\scal  x\xi }\, dx
$$
when $f\in L^1(\rr d)$. Here $\scal \cdo \cdo$ denotes the
usual scalar product on $\rr d$. The map $\mathscr F$ extends 
uniquely to homeomorphisms on $\mathscr S'(\rr d)$, $\mathcal
S_s'(\rr d)$ and $\Sigma _s'(\rr d)$, and restricts to 
homeomorphisms on $\mathscr S(\rr d)$, $\mathcal S_s(\rr d)$
and $\Sigma _s(\rr d)$, and to a unitary operator on $L^2(\rr d)$.
More generally, $\mathscr F$ extends  uniquely to
homeomorphisms from $(\mathcal S_t^s)'(\rr d)$ and $(\Sigma
_t^s)'(\rr d)$ to $(\mathcal S_s^t)'(\rr d)$ and $(\Sigma _s^t)'(\rr d)$
respectively, and restricts to homeomorphisms from $\mathcal
S_t^s(\rr d)$ and $\Sigma _t^s(\rr d)$ to $\mathcal S_s^t(\rr d)$
and $\Sigma _s^t(\rr d)$ respectively.

\par

The following lemma shows that functions in Gelfand-Shilov spaces can
be characterized by estimates on the functions and their Fourier transform
of the form
\begin{equation}\label{GSexpcond}
|f(x)|\lesssim e^{-\ep |x|^{1/t}}\quad \text{and}\quad |\widehat f
(\xi )|\lesssim e^{-\ep |\xi |^{1/s}} .
\end{equation}
The proof is omitted, since the result can be found in e.{\,}g.
\cite{GS, ChuChuKim}.

\par

\begin{lemma}\label{GSFourierest}
Let $f\in \mathcal S'_{1/2}(\rr d)$, and let $s,t>0$. Then the following is true:
\begin{enumerate}
\item if $s+t\ge 1$, then $f\in \mathcal
S_t^s(\rr d)$, if and only if \eqref{GSexpcond} holds
for some $\ep >0$;

\vrum

\item  if $s+t\ge 1$ and $(s,t)\neq (1/2,1/2)$, then $f\in \Sigma _t^s(\rr d)$,
if and only if \eqref{GSexpcond} holds
for every $\ep >0$.
\end{enumerate}
\end{lemma}

\par

The estimates \eqref{GSexpcond} are equivalent to
$$
|f(x)|\le C e^{-\ep |x|^{1/t}}\quad \text{and}\quad |\widehat f
(\xi )|\le Ce^{-\ep |\xi |^{1/s}} .
$$
In (2) in Lemma \ref{GSFourierest}, it is understood that
the (hidden) constant $C>0$ depends on $\ep >0$.

\medspace

Next we recall related characterizations of Gelfand-Shilov
spaces, in terms of short-time Fourier transforms.
(See Propositions \ref{stftGelfand2} and \ref{stftGelfand2dist}
below.)

\medspace

\par

Let $\phi \in \mathscr S(\rr d)\setminus 0$ be fixed. For every $f\in
\mathscr S'(\rr d)$, the \emph{short-time Fourier transform} $V_\phi
f$ is the distribution on $\rr {2d}$ defined by the formula
\begin{equation}\label{defstft}
(V_\phi f)(x,\xi ) =\mathscr F(f\, \overline{\phi (\cdo -x)})(\xi ).
\end{equation}
We note that the right-hand side defines an element in $\mathscr
S'(\rr {2d})\bigcap C^\infty (\rr {2d})$, and 
that $V_\phi f$ takes the form
\begin{equation}\tag*{(\ref{defstft})$'$}
V_\phi f(x,\xi ) =(2\pi )^{-d/2}\int _{\rr d}f(y)\overline {\phi
(y-x)}e^{-i\scal y\xi}\, dy
\end{equation}
when $f\in L^q_{(\omega )}$ for some $\omega \in
\mascP (\rr d)$.

\par

In order to extend the definition of the short-time Fourier
transform we reformulate \eqref{defstft}
in terms of partial Fourier transforms and tensor products
(cf. \cite{Fo}). More precisely, let $\mathscr F_2F$
be the partial Fourier transform of $F(x,y)\in \mathscr S'(\rr {2d})$
with respect to the $y$-variable,
and let $U$ be the map which takes $F(x,y)$ into $F(y,y-x)$.
Then it follows that
\begin{equation}\label{tensorsftf}
V_\phi f =(\mathscr F_2\circ U)(f\otimes \overline \phi )
\end{equation}
when $f\in \mathscr S'(\rr d)$ and $\phi \in \mathscr S(\rr d)$.

\par

The following result concerns the map
\begin{equation}\label{stftmap}
(f,\phi )\mapsto V_\phi f.
\end{equation}

\par

\begin{prop}\label{stftGelfand1}
The map \eqref{stftmap} from $\mathscr S(\rr d)\times \mathscr
S(\rr d)$ to $\mathscr S'(\rr {2d})$ is uniquely extendable to a continuous
map from $\mathcal S_{1/2}'(\rr d)\times \mathcal S'_{1/2}(\rr d)$ to
$\mathcal S'_{1/2}(\rr {2d})$. Furthermore, if $s\ge 1/2$ and
$f,\phi \in \mathcal S'_{1/2}(\rr d)\setminus 0$,
then the following is true:
\begin{enumerate}
\item the map \eqref{stftmap} restricts to a continuous map from
$\mathcal S_{s}(\rr d)\times \mathcal S_{s}(\rr d)$ to $\mathcal
S_{s}(\rr {2d})$. Moreover, $V_\phi f\in \mathcal S_{s}(\rr {2d})$,
if and only if $f,\phi \in \mathcal S_{s}(\rr d)$;

\vrum

\item the map \eqref{stftmap} restricts to a continuous map from
$\mathcal S_{s}'(\rr d)\times \mathcal S_{s}'(\rr d)$ to $\mathcal
S_{s}'(\rr {2d})$. Moreover, $V_\phi f\in \mathcal S_{s}'(\rr {2d})$,
if and only if $f,\phi \in \mathcal S_{s}'(\rr d)$.
\end{enumerate}

\par

Similar facts hold after $\mathcal S_s$ and $\mathcal S_s'$
are replaced by $\Sigma _s$ and $\Sigma _s'$, respectively.
\end{prop}

\par

\begin{proof}
The result follows immediately from
\eqref{tensorsftf}, and the facts that tensor products, $\mathscr F_2$
and $U$ are continuous on $\mathcal S _s$, $\Sigma _s$ and their
duals. See also \cite{CPRT10} for details.
\end{proof}

\par

We also recall characterizations of Gelfand-Shilov spaces and their
distribution spaces
in terms of the short-time Fourier transform, obtained in \cite{GZ,Toft8}.
The involved conditions are
\begin{equation}\label{sts0t0assump}
s\ge s_0>0,\quad t\ge t_0>0\quad \text{and}\quad
s_0+t_0\ge 1
\end{equation}
\begin{align}
|V_\phi f(x,\xi )| &\lesssim e^{-\ep (|x|^{1/t}+|\xi |^{1/s})},
\label{stftexpest2}
\\[1ex]
|(\mathscr F(V_\phi f))(\xi ,x)| &\lesssim  e^{-\ep (|x|^{1/t}+|\xi |^{1/s})}
\label{stftexpest3}
\intertext{and}
|V_\phi f(x,\xi )| &\lesssim e^{\ep (|x|^{1/t}+|\xi |^{1/s})}.
\tag*{(\ref{stftexpest2})$'$}
\end{align}

\par

\begin{prop}\label{stftGelfand2}
Let $s,t,s_0,t_0\in \mathbf R$ satisfy \eqref{sts0t0assump}, and let
$\phi \in \mathcal S_{t_0}^{s_0}(\rr d)\setminus 0$ and $f\in
(\mathcal S_{t_0}^{s_0})'(\rr d)$.
Then the following is true:
\begin{enumerate}
\item $f\in  \mathcal S_{t}^s(\rr d)$, if and only if
\eqref{stftexpest2} holds for some $\ep > 0$.
Furthermore, if $f\in  \mathcal S_{t}^s(\rr d)$, then
\eqref{stftexpest3} holds for some $\ep > 0$;

\vrum

\item if in addition $(s,t)\neq (1/2,1/2)$ and $\phi \in \Sigma _{t}^{s}(\rr d)$, then $f\in
\Sigma _{t}^s(\rr d)$, if and only if \eqref{stftexpest2} holds
for every $\ep > 0$. Furthermore, if $f\in  \mathcal
S_{t}^s(\rr d)$, then \eqref{stftexpest3} holds for every $\ep > 0$.
\end{enumerate}
\end{prop}

\par

We refer to \cite[Theorem 2.7]{GZ} for the proof of Theorem \ref{stftGelfand2}.
The corresponding result for Gelfand-Shilov distributions is the following., and
refer to \cite[Theorem 2.5]{Toft8} for the proof. Note that there
is a misprint in the second statement \cite[Theorem 2.5]{Toft8},
where it stays $f\in  \Sigma _{t}^s(\rr d)$ instead of $f\in
(\Sigma _{t}^s)'(\rr d)$.

\par

\begin{prop}\label{stftGelfand2dist}
Let $s,t,s_0,t_0\in \mathbf R$ satisfy \eqref{sts0t0assump} and $(s,t)\neq (1/2,1/2)$,
and let $\phi \in \Sigma _{t}^s(\rr d)\setminus 0$ and
$f\in (\mathcal S_{t_0}^{s_0})'(\rr d)$. Then the following is true:
\begin{enumerate}
\item $f\in  (\mathcal S_{t}^s)'(\rr d)$, if and only if \eqref{stftexpest2}$'$
holds for every $\ep > 0$;

\vrum

\item $f\in  (\Sigma _{t}^s) '(\rr d)$, if and only if \eqref{stftexpest2}$'$
holds for some $\ep > 0$.
\end{enumerate}
\end{prop}

\par

There are several other ways to characterize Gelfand-Shilov
spaces. For example, they can easily be characterized by
Hermite functions (cf. e.{\,}g. \cite{GrPiRo}).

\par

\subsection{Mixed quasi-normed space of Lebesgue types}

\par

Let $p,q\in (0,\infty ]$, and let $\omega \in \mascP _E(\rr {2d})$.
A common type of of mixed quasi-norm space on $\rr {2d}$ is
$L^{p,q}_{(\omega )}(\rr {2d})$, which consists
of all measurable functions $F$ on $\rr {2d}$ such that
$$
\nm g{L^q(\rr d)}<\infty ,
\quad \text{where}\quad
g(\xi ) \equiv \nm {F(\cdo ,\xi )\omega (\cdo ,\xi )}{L^p(\rr d)}.
$$

\par

Next we introduce a broader family of mixed quasi-norm spaces on
$\rr d$, where the pair $(p,q)$ above
is replaced by a vector in $(0,\infty ]^d$ of Lebesgue exponents.
If
$$
\mabfp =(p_1,\dots , p_d)\in (0,\infty ]^d
\quad \text{and}\quad
\mabfq =(q_1,\dots , q_d)\in (0,\infty ]^d
$$
are two such vectors, then we use the conventions $\mabfp \le \mabfq$
when $p_j\le q_j$ for every $j=1,\dots ,d$, and $\mabfp < \mabfq$
when $p_j< q_j$ for every $j=1,\dots ,d$.

\par

Let $\operatorname {S}_d$ be the set of permutations
on $\{ 1,\dots ,d\}$, $\mabfp \in
(0,\infty ]^d$, $\omega \in \mascP _E(\rr d)$, and let $\sigma
\in \operatorname {S}_d$.
For every measurable and complex-valued function $f$ on
$\rr d$, let
$g_{j,\omega }$, $j=1,\dots ,d$, be defined inductively by
the formulas
\begin{align}
g_\omega (x_{\sigma (1)},\dots ,x_{\sigma (d)}) &\equiv |f(x_1,\dots ,x_d)
\omega (x_1,\dots ,x_d)|,\label{gomegadef}
\\[1ex]
g_{1,\omega}(x_2,\dots ,x_d) &\equiv \nm {g_\omega (\cdo ,x_2,\dots ,x_d) }
{L^{p_1}(\mathbf R)},\notag
\\[1ex]
g_{k,\omega }(x_{k+1},\dots ,x_d) &\equiv \nm {g_{k-1,\omega}(\cdo ,
x_{k+1},\dots ,x_d) }
{L^{p_k}(\mathbf R)},
\quad k=2,\dots ,d-1 .\notag
\intertext{and}
\nm f{L^{\mabfp}_{\sigma ,(\omega )}} \equiv g_{d,\omega} &\equiv
\nm {g_{d-1,\omega}}{L^{p_d}(\mathbf R)}.\notag
\end{align}
The mixed quasi-norm space $L^{\mabfp}_{\sigma ,(\omega )}(\rr d)$ of
Lebesgue type is defined as the set of all complex-valued measurable functions
$f$ on $\rr d$ such that $\nm f{L^{\mabfp }_{\sigma ,(\omega )}}<\infty$.

\par

The set of sequences $\ell ^{\mabfp }_{\sigma ,(\omega )}(\Lambda)$,
for an appropriate lattice $\Lambda$ is defined in
an analogous way. More precisely, let
$\theta =(\theta _1,\dots ,\theta _d)\in \rrstar d$, and let
$T_\theta$ denote the diagonal matrix with diagonal elements
$\theta _1,\dots ,\theta _d$. Here $\mathbf R_*=\mathbf
R\setminus 0$ and we interprete $\rrstar d$ as $(\mathbf
R\setminus 0)^d$. Also let
\begin{equation}\label{LambdaDef}
\Lambda = T_\theta \zz d \equiv
\sets {(\theta _1j_1,\dots ,\theta _dj_d)}{(j_1,\dots ,j_d)\in \zz d} .
\end{equation}

\par

For any sequence $a$ on $T_\theta \zz d$, let
$b_{j,\omega }$, $j=1,\dots ,d$, be defined inductively by the formulas
\begin{align}
b_\omega (j_{\sigma (1)},\dots ,j_{\sigma (d)}) &\equiv |a(j_1,\dots ,j_d)
\omega (j_1,\dots ,j_d)|,\label{bomegadef}
\\[1ex]
b_{1,\omega}(j_2,\dots ,j_d) &\equiv \nm {b_\omega (\cdo ,j_2,\dots ,j_d) }
{\ell ^{p_1}(\theta _1\mathbf Z)}\notag
\\[1ex]
b_{k,\omega }((j_{k+1},\dots ,j_d) &\equiv
\nm {b_{k-1,\omega}(\cdo ,j_{k+1},\dots ,j_d) }
{\ell ^{p_k}(\theta _k\mathbf Z)},\quad k=2,\dots ,d-1 \notag
\intertext{and}
\nm a{\ell ^{\mabfp}_{\sigma ,(\omega )}(\Lambda )}
\equiv b_{d,\omega} &\equiv \nm {b_{d-1,\omega}}{\ell ^{p_d}(\theta
_d\mathbf Z)}.\notag
\end{align}
The mixed quasi-norm space
$\ell ^{\mabfp}_{\sigma ,(\omega )}(\Lambda )$
is defined as the set of all sequences functions
$a$ on $\Lambda$ such that
$\nm a{\ell ^{\mabfp}_{\sigma ,(\omega )}(\Lambda )}<\infty$.

\par

We also write $L^{\mabfp }_{(\omega )}$ and $\ell ^{\mabfp }
_{(\omega )}$ instead of $L^{\mabfp }_{\sigma ,(\omega )}$
and $\ell ^{\mabfp } _{\sigma ,(\omega )}$ respectively when $\sigma$
is the identity map. Furthermore, if $\omega$ is equal to $1$, then we write
\begin{alignat*}{5}
&L^{\mabfp }_{\sigma }, &
\quad
&\ell ^{\mabfp } _{\sigma }, &
&L^{\mabfp } &
\quad &\text{and} & \quad
&\ell ^{\mabfp }
\intertext{instead of}
&L^{\mabfp }_{\sigma ,(\omega )}, &
\quad
&\ell ^{\mabfp } _{\sigma ,(\omega )}, &\quad 
&L^{\mabfp }_{(\omega )} &
\quad &\text{and} & \quad
&\ell ^{\mabfp } _{(\omega )},
\end{alignat*}
respectively.

\par

For any $\mabfp \in (0,\infty ]^d$, let
$$
\max \mabfp \equiv \max (p_1,\dots ,p_d)
\quad \text{and}\quad
\min \mabfp \equiv \min (p_1,\dots ,p_d).
$$
We note that if $\max \mabfp <\infty$, then $\ell _0 (\Lambda )$ is dense in
$\ell ^\mabfp _{\sigma ,(\omega )} (\Lambda )$. Here $\ell _0
(\Lambda )$ is the set of all sequences $\{ a(j)\} _{j\in \Lambda}$
on $\Lambda$ such that $a(j)\neq 0$ for at most finite numbers of
$j$.

\par

\subsection{Modulation spaces}\label{subsec1.2}

\par

Next we define modulation spaces.
Let $\phi \in \maclS _{1/2}(\rr d)\setminus 0$. For any $p,q\in (0.\infty ]$
and $\omega \mascP _E(\rr {2d})$,
the \emph{standard} modulation space $M^{p,q}_{(\omega )}(\rr d)$
is the set of all $f\in \maclS _{1/2}'(\rr d)$ such that $V_\phi f\in
L^{p,q}_{(\omega )}(\rr {2d})$, and we equip $M^{p,q}_{(\omega )}(\rr d)$
with the quasi-norm
$$
\nm f{M^{p,q}_{(\omega )}}\equiv \nm {V_\phi f}{L^{p,q}_{(\omega )}}.
$$
We remark that $M^{p,q}_{(\omega )}(\rr d)$ is one of the
most common types of modulation spaces.

\par

More generally, for any $\sigma \in \operatorname S_{2d}$,
$\mabfp \in (0,\infty ]^{2d}$ and $\omega \in \mascP _E(\rr {2d})$,
the modulation space $M^\mabfp _{\sigma ,(\omega )}(\rr d)$
is the set of all $f\in \maclS _{1/2}'(\rr d)$ such that $V_\phi f\in
L^{\mabfp}_{\sigma ,(\omega )}(\rr {2d})$, and we equip
$M^{\mabfp}_{\sigma (\omega )}(\rr d)$ with the quasi-norm
\begin{equation}\label{modnorm2}
\nm f{M^{\mabfp}_{\sigma , (\omega )} }\equiv
\nm {V_\phi f}{L^{\mabfp}_{\sigma ,(\omega )}}.
\end{equation}

\par

In the following propositions we list some properties for modulation.
The first one follows from the definition of invariant spaces and
Propositions \ref{stftGelfand2} and \ref{stftGelfand2dist}. The other
results can be found in \cite{Fei1,FG1,FG2,Gc2,Toft5}. The proofs are
therefore omitted

\par

\begin{prop}\label{p1.4A}
Let $\omega \in \mascP  _{E}(\rr {2d})$, $\sigma \in \operatorname S_{2d}$ and
$\mabfp \in (0,\infty ]^{2d}$. Then the following is true:
\begin{enumerate}
\item[{\rm{(1)}}] $\Sigma _1(\rr d)\subseteq M^{\mabfp}_{\sigma ,(\omega )}(\rr d)
\subseteq \Sigma _1'(\rr d)$;

\vrum

\item[{\rm{(2)}}] if in addition
$$
e^{-\ep |\cdo |}\lesssim \omega \lesssim e^{\ep |\cdo |},
$$
holds for every $\ep >0$, then
$\maclS _1(\rr d)\subseteq M^{\mabfp}_{\sigma ,(\omega )}(\rr d)
\subseteq \maclS _1'(\rr d)$;

\vrum

\item[{\rm{(3)}}]  if in addition $\omega \in \mascP (\rr {2d})$, then
$\mathscr S(\rr d)\subseteq M^{\mabfp}_{\sigma ,(\omega )}(\rr d)
\subseteq \mathscr S '(\rr d)$.
\end{enumerate}
\end{prop}

\par

\begin{prop}\label{p1.4B}
Let
$$
p,q\in [1,\infty ],\quad \mabfp ,\mabfp _j\in [1,\infty ] ^{2d},
\quad
\omega ,\omega _j,v\in \mascP  _{E}(\rr {2d}),
\quad j=1,2,
$$
be such that
$\mabfp _1\le \mabfp _2$,  $\omega _2\lesssim \omega _1$, and
$\omega$ is $v$-moderate. Also
let $\sigma \in \operatorname S_{2d}$. Then the following is true:
\begin{enumerate}
\item if $\phi \in M^1_{(v)}(\rr d)\setminus 0$, then
$f\in M^{\mabfp}_{\sigma ,(\omega )}(\rr d)$, if and only if
$$
\nm {V_\phi f} {L^{\mabfp}_{\sigma ,(\omega )}}<\infty .
$$
In particular, $M^{\mabfp}_{\sigma ,(\omega )}(\rr d)$ is independent
of the choice of $\phi \in M^1_{(v)}(\rr d)\setminus 0$.
Moreover, $M^{\mabfp}_{\sigma ,(\omega )}(\rr d)$ is a Banach
space under the norm in \eqref{modnorm2}, and different
choices of $\phi$ give rise to equivalent norms;

\vrum

\item[{\rm{(2)}}] $M^{\mabfp _1}_{\sigma ,(\omega _1)}(\rr d)\subseteq
M^{\mabfp _2}_{\sigma ,(\omega _2)}(\rr d)$;

\vrum

\item[{\rm{(3)}}] the $L^2$-form on $\maclS _{1/2}(\rr d)$ extends uniquely
to a dual form between $M^{p,q}_{(\omega )}(\rr d)$ and
$M^{p',q'}_{(1/\omega )}(\rr d)$. Furthermore, if in addition $p,q<\infty$,
then the dual of $M^{p,q}_{(\omega )}$ can be identified with
$M^{p',q'}_{(1/\omega )}(\rr d)$ through this form.
\end{enumerate}
\end{prop}

\par

Next we recall the notion of Gabor expansions. First we recall some
facts on sequences and lattices. In what follows we let $\Lambda$,
$\Lambda _1$ and $\Lambda _2$ be the lattices 
\begin{equation}\label{Lattices}
\Lambda _1 \equiv \{ x_j \} _{j\in J}  \equiv T_\theta \zz d ,
\quad
\Lambda _2 \equiv \{ \xi _k \} _{k\in J} \equiv  T_\vartheta \zz d ,
\quad
\Lambda \equiv \Lambda _1\times \Lambda _2
\end{equation}
where $\theta ,\vartheta \in \rrstar d$, and $J$ is an index set.

\par

\begin{defn}\label{DefAnSynGabOps}
Let $\Lambda$, $\Lambda _1$
and $\Lambda _2$ be as in \eqref{Lattices}. Let
$\omega ,v\in \mascP _E(\rr {2d})$ be such that
$\omega$ is $v$-moderate, and let $\phi ,\psi \in M^1_{(v)}(\rr d)$.
\begin{enumerate}
\item The analysis operator $C^{\Lambda}_\phi$ is the operator from
$M^\infty _{(\omega )}(\rr d)$ to $\ell ^{\infty}_{(\omega)}(\Lambda )$,
given by
$$
C^\Lambda _\phi f \equiv \{ V_\phi f(x_j,\xi _k) \} _{j,k\in J} \text ;
$$

\vrum

\item The synthesis operator $D^{\Lambda}_\psi$ is the operator from
$\ell ^\infty _{(\omega )}(\Lambda )$ to $M^\infty _{(\omega)}(\rr d)$,
given by
$$
D^\Lambda _\psi c \equiv \sum _{j,k\in J} c_{j,k}
e^{i\scal \cdo {\xi _k}}\phi (\cdo -x_j)\text ;
$$

\vrum

\item The Gabor frame operator $S^{\Lambda}_{\phi ,\psi}$
is the operator on $M^\infty _{(\omega )}(\rr d)$,
given by $D^\Lambda _\psi \circ
C^\Lambda _\phi$, i.{\,}e.
$$
S^{\Lambda}_{\phi ,\psi}f \equiv \sum _{j,k\in J} V_\phi f(x_j,\xi _k)
e^{i\scal \cdo {\xi _k}}\psi (\cdo -x_j).
$$
\end{enumerate}
\end{defn}

\par

It follows from the analysis in Chapters 11--14 in \cite{Gc2}
that the operators in Definition \ref{DefAnSynGabOps} are
well-defined and continuous.

\par

We finish the section by discussing some consequences of the following
result. The proof is omitted since the result follows from Theorem 13.1.1
in \cite{Gc2}, which in turn can be considered as a special case of Theorem
S in \cite{Gc1}.

\par

\begin{prop}\label{ThmS}
Let $v\in \mascP _E(\rr {2d})$ be submultiplicative, and $\phi \in
M^1_{(v)}(\rr d)\setminus 0$.
Then there is a constant $\ep _0>0$ such that for every $\ep \in
(0,\ep _0]$, the frame operator $S_{\phi ,\phi}^\Lambda$, with
$\Lambda =\ep \zz{2d}$ is a homeomorphism on $M^1_{(v)}(\rr d)$.
\end{prop}

\par

We also recall the following result, and refer to the proof of Corollaries
12.2.5 and 12.2.6 in \cite{Gc2} for the proof.

\par

\begin{prop}
Let $v$, and $\phi$ and $\Lambda$ be the same as in Proposition
\ref{ThmS}, $\psi =(S_{\phi ,\phi}^\Lambda )^{-1}\phi$,
$f\in M^\infty _{(1/v)}(\rr d)$, $p,q\in [1,\infty ]$,and let
$\omega \in \mascP _E(\rr {2d})$ be $v$-moderate. Then
\begin{equation}\label{GabExp}
\begin{aligned}
f &= \sum _{(x_j,\xi _k)\in \Lambda} V_\phi f(x_j,\xi _k)
e^{i\scal \cdo {\xi _k}}\psi (\cdo -x_j)
\\[1ex]
&= \sum _{(x_j,\xi _k)\in \Lambda}
V_\psi f(x_j,\xi _k) e^{i\scal \cdo {\xi _k}}\phi (\cdo -x_j) ,
\end{aligned}
\end{equation}
where the sums converge in the weak$^*$ topology. Furthermore the
following conditions are equivalent.
\begin{enumerate}
\item $f\in M^{p,q}_{(\omega )}(\rr d)$;

\vrum

\item $\{ V_\phi f(x_j,\xi _k) \} _{(x_j,\xi _k)\in \Lambda} \in
\ell ^{p,q}_{(\omega )}(\Lambda )$;

\vrum

\item $\{ V_\psi f(x_j,\xi _k) \} _{(x_j,\xi _k)\in \Lambda} \in
\ell ^{p,q}_{(\omega )}(\Lambda )$.
\end{enumerate}

\par

Moreover, if {\rm{(1)}}--{\rm{(3)}} are true, then the sums in \eqref{GabExp}
converge in the weak$^*$ topology to elements in $M^{p,q}_{(\omega )}$
when $p=\infty$ or $q=\infty$, and unconditionally in norms when $p,q<\infty$.
\end{prop}

\par

Let $v$, $\phi$ and $\Lambda$ be as in Proposition \ref{ThmS}. Then
$$
(S_{\phi ,\phi}^\Lambda )^{-1}\phi
$$
is called the \emph{canonical dual window} to $\phi$, with respect to
$\Lambda$. By duality, it follows that $S_{\phi ,\phi}^\Lambda$
extends to to a continuous operator on $M^\infty _{(1/v)}(\rr d)$, and
$$
S_{\phi ,\phi}^\Lambda (e^{i\scal \cdo {\xi _k}}f(\cdo -x_j)) =
e^{i\scal \cdo {\xi _k}}(S_{\phi ,\phi}^\Lambda f)(\cdo -x_j),
$$
when $f\in M^\infty _{(1/v)}(\rr d)$ and $(x_j,\xi _k)\in \Lambda$.
The series in \eqref{GabExp}
are called \emph{Gabor expansions} of $f$ with respect to $\phi$ and
$\psi$.

\par

Now let $\mabfp =[1,\infty ]^{2d}$, $\sigma \in \operatorname S_{2d}$,
and let $\omega ,v\in \mascP _E(\rr {2d})$ be such that $\omega$ is
$v$-moderate, and choose $\phi$ and $\ep _0$ such that the conclusions in
Proposition \ref{ThmS} are true. Also let $f\in M^{\mabfp}
_{\sigma, (\omega )}(\rr d)$. Then the right-hand sides of \eqref{GabExp}
converge unconditionally in $M^{\mabfp} _{\sigma ,(\omega )}$ when $\max
\mabfp <\infty$, and in $M^\infty _{(\omega )}$ with respect
to the weak$^*$ topology when $\max \mabfp =\infty$.
(Cf. \cite{FG1,Gc2}.) For modulation spaces of the form
$M^{p,q}_{(\omega )}$ with $\omega$ belonging to the subset
$\mascP$ of $\mascP _E$, these properties were extended in \cite{GaSa} to
the quasi-Banach case, allowing $p$ and $q$ to be smaller than $1$.
In Section \ref{sec3} we extend all these properties to more
general $M^\mabfp _{\sigma ,(\omega )}$, where $\sigma \in
\operatorname S_{2d}$, $\omega \in \mascP _E$ and
$\mabfp \in (0,\infty ]^{2d}$, based on the analysis in \cite{GaSa}.

\par

\begin{rem}\label{RemThmS}
Let $r\in (0,1)$, $v\in \mascP _E(\rr {2d})$ be submultiplicative, and
set
\begin{equation}\label{SubMultMod}
(\Theta _\rho v)(x,\xi )=v(x,\xi )\eabs {x,\xi}^\rho ,
\quad \text{where}\quad \rho >2d(1-r)/r.
\end{equation}
Then $L^1_{(\Theta _\rho v)}(\rr {2d})$ is continuously embedded in
$L^r_{(v)}(\rr {2d})$, giving that $M^1_{(\Theta _\rho v)}(\rr d)
\subseteq M^r_{(v)}(\rr d)$. Hence if $\phi \in M^1_{(\Theta _\rho v)}
\setminus 0$, $\ep _0$ is chosen such that $S^\Lambda _{\phi ,\phi}$
is invertible on $M^1_{(\Theta _\rho v)}(\rr d)$ for every
$\Lambda =\ep \zz {2d}$, $\ep \in (0,\ep _0]$, it follows that both
$\phi$ and its canonical dual with respect to $\Lambda$ belong to
$M^r_{(v)}(\rr d)$.
\end{rem}

\par

\section{Convolution estimates for Lebesgue and Wiener
spaces}\label{sec2}

\par

In this section we deduce continuity properties for discrete, semi-discrete
and non-discrete convolutions. Especially we discuss such mapping
properties for sequence and Wiener spaces.

\par

In what follows we let $T_\theta$, $\theta \in \rrstar d$ be the diagonal
$d\times d$-matrix, with $\theta _1,\dots ,\theta _d$ as diagonal
values as in the previous section. The semi-discrete
convolution with respect to $\theta$ is given by
$$
(a*_{[\theta ]}f)(x) \sum _{j\in \zz d}a(j)f(x - T_\theta j),
$$
when $f\in \maclS _{1/2}'(\rr d)$ and $a \in \ell _0 (\zz d)$. 

\par

We have the following proposition.

\par

\begin{prop}\label{PropSemiContConvEst}
Let $\sigma \in \operatorname {S}_d$, $\omega ,v\in \mascP _E(\rr d)$ be
such that $\omega$ is $v$-moderate, $\theta ,\vartheta \in \rrstar d$, and let
$\mabfp ,\mabfr \in (0,\infty ]^d$ be such that $\theta _j = \vartheta _{\sigma (j)}$,
$j=1,\dots , d$, and 
$$
r_k\le \min _{m\le k}(1,p_m).
$$
Also let $v_\theta = v\circ T_\theta$.
Then the map $(a,f)\mapsto a*_{[\vartheta ]}f$ from $\ell _0(\zz d)\times
\maclS _{1/2}(\rr d)$ to $\maclS _{1/2}(\rr d)$ extends uniquely to a
linear and continuous map from $\ell ^{\mabfr}_{\sigma ,(v_\theta )}(\zz d)
\times L^{\mabfp}_{\sigma ,(\omega )}(\rr d)$ to $L^{\mabfp}
_{\sigma ,(\omega )}(\rr d)$, and
\begin{equation}\label{convest1}
\nm {a*_{[\vartheta ]}f}{L^{\mabfp}_{\sigma ,(\omega )}}\le
C\nm {a}{\ell ^{\mabfr}_{\sigma ,(v_\theta )}(\zz d)}\nm {f}{L^{\mabfp}
_{\sigma ,(\omega )}},
\end{equation}
where the constant $C$ is the same as in \eqref{moderate}.
\end{prop}

\par

\begin{proof}
We only consider the case $\max \mabfp <\infty$. The modifications
to the case when at least one $p_j$ equals $\infty$ is straight-forward
and is left for the reader.

\par

Let $h$ be defined by
$$
h_\omega (x_{\sigma (1)},\dots ,x_{\sigma (d)}) = (|a| *
_{[\vartheta ]} |f|)(x)\omega (x).
$$
Then it follows by straight-forward computations that
$$
h_\omega \le C|b_{v_\theta } |*_{[\theta ]}|g_\omega |,
$$
where $b_{v_\theta }$ is given by \eqref{bomegadef} with
$\omega =v_\theta $ and $\Lambda =\zz d$, and $g_\omega $
is given by \eqref{gomegadef}. Since
\begin{gather*}
\nm a{\ell ^\mabfr _{\sigma ,(v_\theta )} (\zz d)} =
\nm {b_{v_\theta}}{\ell ^\mabfr (\zz d)},
\quad
\nm f{L^\mabfp _{\sigma ,(\omega )}} = \nm {g_\omega}{L^\mabfp } 
\intertext{and}
\nm {a*_{[\vartheta ]}f}{L^\mabfp _{\sigma ,(\omega )}}
\le \nm {h_\omega}{L^\mabfp },
\end{gather*}
it follows that we may assume that $f$ and $a$ are non-negative,
$\sigma$ is the identity map and that $\omega = v = 1$, giving
that $\vartheta =\theta$.

\par

For $x\in \rr d$ and $j\in \zz d$, we let
$$
y_k = (x_{k+1}, \dots ,x_d)\in \rr {d-k},\quad
l_k = (j_{k+1},\dots ,j_d)\in \zz {d-k},\quad
$$
$k=1,\dots ,d-1$. Also let
\begin{align*}
F_x(j) &= f(x+T_\theta j),\quad G_{1,x}(l_1) = G_1(x,l_1)
\equiv \nm {F_x(\cdo ,l_1)}{\ell ^{p_1}(\mathbf Z)},
\\[1ex]
F_{1,y_1}(l_1) &= F_1(y_1,l_1) \equiv \nm {G_1(\cdo ,y_1,l_1)}
{L^{p_1}[0,\theta _1]},
\intertext{and define inductively}
G_{k,y_{k-1}}(l_k) &= G_k(y_{k-1},l_k)
\equiv \nm {F_{k-1,y_{k-1}}(\cdo ,l_k)}{\ell ^{p_k}(\mathbf Z)},
\\[1ex]
F_{k,y_k}(l_k) &= F_k(y_k,l_k) \equiv \nm {G_k(\cdo ,y_k,l_k)}
{L^{p_k}[0,\theta _k]}, \quad k=2,\dots ,d-1,
\\[1ex]
G_d(x_d) &\equiv \nm {F_{d-1,x_d}}{\ell ^{p_d}(\mathbf Z)},
\quad \text{and}\quad 
F_d \equiv \nm {G_d}{L^{p_d}[0,\theta _d]}. 
\end{align*}
In the same way, let
\begin{align*}
A_1(l_1) &\equiv \nm {a(\cdo ,l_1)}{\ell ^{r_1}(\mathbf Z)}
\\[1ex]
A_k(l_k) &= \nm {A_{k-1}(\cdo ,l_k)}
{\ell ^{r_k}(\mathbf Z)},\quad k=2,\dots ,d-1,
\intertext{and}
A_d &= \nm {A_{d-1}}{\ell ^{r_d}(\mathbf Z)},
\end{align*}
Finally, let $H_{k,y_k}(l_k) = H_k (y_k ,l_k)$ and $H_d$
be the same as $F_{k,y_k}(l_k)$ and $F_d$, respectively,
$k=1,\dots ,d-1$, after
$f$ has been replaced by $a*_{[\theta ]}f$. By straight-forward
computations it follows that
$$
\nm {f}{L^{\mabfp}} = F _d,\quad \nm {a*_{[\theta ]}f}
{L^{\mabfp}} = H _d\quad  \text{and}\quad
\nm {a}{\ell ^{\mabfr}} = A_d.
$$

\par

We claim that
\begin{equation}\label{HkykEst}
\begin{aligned}
H_{k,y_k}(l_k) &\le \big ( (A _k^{r_k}*_{[\vartheta _k]}
F_{k,y_k}^{r_k})(l_k) \big ) ^{1/r_k},
\quad k=1,\dots ,d-1,
\\[1ex]
H_d&\le A_dF_d,
\end{aligned}
\end{equation}
where $\vartheta _k = (\theta _{k+1},\dots ,\theta _d)$.

\par

In fact, first assume that $k=1$. We have
\begin{align}
H_{1,y_1}(l_1) &= \left ( \int _0^{\theta _1} J(x_1,y_1,l_1)\, dx_1
\right )^{1/p_1},\label{H1formula}
\intertext{where}
J(x_1,y_1,l_1) &= \sum _{j_1\in \mathbf Z} \left (
\sum _{m\in \zz {d-1}} \big (a(\cdo ,m) *_{\theta _1}
F_x(\cdo ,T_{\vartheta _1}(l_1-m))\big )(j_1) \right )^{p_1}.\notag
\end{align}
Here $*_{\theta _k}=*_{[\theta _k]}$ denotes the one-dimensional
semi-discrete convolution with respect to $\theta _k$.
We shall consider the cases $p_1\ge 1$ and $p_1<1$
separately, and start to consider the former one.

\par

Therefore, assume that $p_1\ge 1$. By applying Minkowski's
inequality on $J(x_1,y_1,l_1)$ we get
\begin{multline*}
\sum _{j_1\in \mathbf Z} \left (
\sum _{m\in \zz {d-1}} \big (a(\cdo ,m) *_{\theta _1}
F_x(\cdo ,T_{\vartheta _1}(l_1-m))\big )(j_1) \right )^{p_1}
\\[1ex]
\le
\left (
\sum _{m\in \zz {d-1}} \nm {a(\cdo ,m) *_{\theta _1}
F_x(\cdo ,T_{\vartheta _1}(l_1-m))}{\ell ^{p_1}} \right )^{p_1}
\\[1ex]
\le
\left (
\sum _{m\in \zz {d-1}} A_1(m)G_{1}(x,T_{\vartheta _1}(l_1-m))
\right )^{p_1}.
\end{multline*}
By using this estimate in \eqref{H1formula} we get
\begin{multline*}
H_{1,y_1}(l_1) \le  \left ( \int _0^{\theta _1} \left (
\sum _{m\in \zz {d-1}} A_1(m)G_{1}(x,T_{\vartheta _1}(l_1-m))
\right )^{p_1}\, dx_1 \right )^{1/p_1}
\\[1ex]
\le \sum _{m\in \zz {d-1}} A_1(m)F_{1}(y_1,T_{\vartheta _1}(l_1-m))
\le (A_1*_{[\vartheta _1]}F_{1,y_1})(l_1),
\end{multline*}
and \eqref{HkykEst} follows in the case $k=1$ and $p_1\ge 1$.

\par

Next assume that $p_1<1$. Then we get
\begin{multline*}
\sum _{j_1\in \mathbf Z} \left (
\sum _{m\in \zz {d-1}} \big (a(\cdo ,m) *_{\theta _1}
F_x(\cdo ,T_{\vartheta _1}(l_1-m))\big )(j_1) \right )^{p_1}
\\[1ex]
\le \sum _{m\in \zz {d-1}}  \sum _{j_1\in \mathbf Z}  
\left ( \big (a(\cdo ,m) *_{\theta _1}
F_x(\cdo ,T_{\vartheta _1}(l_1-m))\big )(j_1) \right )^{p_1}
\\[1ex]
\le
\sum _{m\in \zz {d-1}} A_1(m)^{p_1}G_{1}(x,T_{\vartheta
_1}(l_1-m)) ^{p_1}.
\end{multline*}
By using this estimate in \eqref{H1formula} we get
\begin{multline*}
H_{1,y_1}(l_1) \le  \left ( \int _0^{\theta _1}
\sum _{m\in \zz {d-1}} A_1(m)^{p_1}G_{1}(x,T_{\vartheta _1}(l_1-m))
^{p_1}\, dx_1 \right )^{1/p_1}
\\[1ex]
\le
\left (\sum _{m\in \zz {d-1}} A_1(m)^{p_1}\int _0^{\theta _1}
G_{1}(x,T_{\vartheta _1}(l_1-m))^{p_1}\, dx_1\right ) ^{1/p_1}
\\[1ex]
\le \big ( (A_1^{p_1}*_{[\vartheta _1]}F_{1,y_1}^{p_1})(l_1)\big )^{1/p_1},
\end{multline*}
and \eqref{HkykEst} follows in the case $k=1$ for any
$p_1\in (0,\infty ]$.

\par

Next we assume that \eqref{HkykEst} holds for $k< n$, where
$1\le n\le d$, and prove the result for $k=n$. The relation
\eqref{HkykEst} then follows by induction.

\par

First we consider the case $q_n\equiv p_n/r_{n-1}\ge 1$. Set $y=y_{n-1}
=(x_n,\dots ,x_d))$. Then $r_n=r_{n-1}$, and the inductive
assumption together with Minkowski's and Young's inequalities
give
\begin{multline*}
\sum _{j_n\in \mathbf Z} \left ( H_{n-1,y_{n-1}}(j_n,l_n) \right )^{p_n}
\\[1ex]
\le \sum _{j_n\in \mathbf Z} \left (
\sum _{m\in \zz {d-n}} \big (A_{n-1}(\cdo ,m)^{r_{n-1}} *_{\theta _n}
F_{n-1,x}(\cdo ,T_{\vartheta _n}(l_n-m))\big )(j_n)^{r_{n-1}} \right )^{q_n}
\\[1ex]
\le
\left (
\sum _{m\in \zz {d-n}} \nm {A_{n-1}(\cdo ,m)^{r_{n-1}} *_{\theta _n}
F_{n-1,x}(\cdo ,T_{\vartheta _n}(l_n-m))^{r_{n-1}}}{\ell ^{q_n}}
\right )^{q_n}
\\[1ex]
\le
\left (
\sum _{m\in \zz {d-n}} \nm {A_{n-1}(\cdo ,m)^{r_{n-1}} }{\ell ^1}
\nm {F_{n-1,x}(\cdo ,T_{\vartheta _n}(l_n-m))^{r_{n-1}}}{\ell ^{q_n}}
\right )^{q_n}
\\[1ex]
=
\left (
\sum _{m\in \zz {d-n}} \nm {A_{n-1}(\cdo ,m)} {\ell ^{r_{n-1}}}^{r_{n-1}}
\nm {F_{n-1,x}(\cdo ,T_{\vartheta _n}(l_n-m))}{\ell ^{p_n}}^{r_{n-1}}
\right )^{q_n}
\\[1ex]
\le
\left (
\sum _{m\in \zz {d-n}} A_{n}(m)^{r_{n}}
\nm {F_{n-1,x}(\cdo ,T_{\vartheta _n}(l_n-m))}{\ell ^{p_{n}}}^{r_{n}}
\right )^{p_n/r_n}
\\[1ex]
\le
\left (
\sum _{m\in \zz {d-n}} A_{n}(m)^{r_{n}}
G_{n,x}(T_{\vartheta _n}(l_n-m))^{r_{n}}
\right )^{p_n/r_n},
\end{multline*}
where the last inequality follows from the facts that $r_n=r_{n-1}$ when
$p_n/p_{n-1}\ge 1$. Minkowski's inequality now gives
\begin{multline*}
H_{n,y_{n}}(l_n) = 
\left ( \int _0^{\theta _n}\sum _{j_n\in \mathbf Z} \left ( H_{n-1,y_{n-1}}(j_n,l_n)
\right )^{p_n} \, d{x_n} \right ) ^{1/p_n}
\\[1ex]
\le 
\left ( \int _0^{\theta _n} \left ( \sum _{m\in \zz {d-n}} A_{n}(m)^{r_{n}}
G_{n}(x_n,y_n,T_{\vartheta _n}(l_n-m))^{r_{n}}
\right )^{p_n/r_n}\, d{x_n} \right ) ^{1/p_n}
\\[1ex]
\le
\left ( \sum _{m\in \zz {d-n}} A_{n}(m)^{r_{n}}
\nm {G_{n}(\cdo ,y_n,T_{\vartheta _n}(l_n-m))}{L^{p_n/r_n}[0,\theta _n]}^{r_{n}}
\right )^{1/r_n},
\end{multline*}
which gives \eqref{HkykEst} for $k=n$ in this case.

\par

Next we consider the case when $q_n=p_n/r_{n-1}< 1$.
Then $r_n=p_{n}$, and the inductive assumption together with
Minkowski's inequality, Young's inequality and the fact that $\ell ^{q_n}$
is an algebra under convolution, give
\begin{multline*}
\sum _{j_n\in \mathbf Z} \left ( H_{n-1,y_{n-1}}(j_n,l_n) \right )^{p_n}
\\[1ex]
\le \sum _{j_n\in \mathbf Z} \left (
\sum _{m\in \zz {d-n}} \big (A_{n-1}(\cdo ,m)^{r_{n-1}} *_{\theta _n}
F_{n-1,y}(\cdo ,T_{\vartheta _n}(l_n-m))\big )(j_n)^{r_{n-1}} \right )^{p_n/r_{n-1}}
\\[1ex]
\le  \sum _{m\in \zz {d-n}} \sum _{j_n\in \mathbf Z}
\left ( \big (A_{n-1}(\cdo ,m)^{r_{n-1}} *_{\theta _n}
F_{n-1,y}(\cdo ,T_{\vartheta _n}(l_n-m))\big )(j_n)^{r_{n-1}}
\right )^{p_n/r_{n-1}}
\\[1ex]
\le  \sum _{m\in \zz {d-n}} 
\left ( \nm {A_{n-1}(\cdo ,m)^{r_{n-1}}}{\ell ^{p_n/r_{n-1}}} \nm 
{F_{n-1,y}(\cdo ,T_{\vartheta _n}(l_n-m))^{r_{n-1}}}{\ell ^{p_n/r_{n-1}}}
\right )^{p_n/r_{n-1}}
\\[1ex]
=
\sum _{m\in \zz {d-n}} 
\nm {A_{n-1}(\cdo ,m)}{\ell ^{p_n}} ^{p_n}\nm 
{G_{n,y}(T_{\vartheta _n}(l_n-m))}{\ell ^{p_n}}^{p_n}
\\[1ex]
=
\sum _{m\in \zz {d-n}} 
A_{n}(m)^{r_n}\nm 
{G_{n,(x_n,y_n)}(T_{\vartheta _n}(l_n-m))}{\ell ^{p_n}}^{r_n}.
\end{multline*}
This gives
\begin{multline*}
H_{n,y_{n}}(l_n)
\le
\left ( \int _0^{\theta _n}\sum _{j_n\in \mathbf Z} \left ( H_{n-1,y_{n-1}}(j_n,l_n)
\right )^{p_n} \, d{x_n} \right ) ^{1/p_n}
\\[1ex]
\le 
\left ( \int _0^{\theta _n}
\sum _{m\in \zz {d-n}} 
A_{n}(m)^{r_n}\nm 
{G_{n,(x_n,y_n)}(T_{\vartheta _n}(l_n-m))}{\ell ^{p_n}}^{p_n}
\, d{x_n} \right ) ^{1/p_n}
\\[1ex]
=
\left ( 
\sum _{m\in \zz {d-n}} 
A_{n}(m)^{r_n} \int _0^{\theta _n} \nm 
{G_{n,(x_n,y_n)}(T_{\vartheta _n}(l_n-m))}{\ell ^{p_n}}^{p_n}
\, d{x_n} \right ) ^{1/p_n}
\\[1ex]
=
\left ( 
\sum _{m\in \zz {d-n}} 
A_{n}(m)^{r_n}
F_{n,y_n}(l_n-m)^{p_n} \right ) ^{1/p_n}
\\[1ex]
=
\left ( 
\sum _{m\in \zz {d-n}} 
A_{n}(m)^{r_n}
F_{n,y_n}(l_n-m)^{r_n} \right ) ^{1/r_n}.
\end{multline*}
This gives \eqref{HkykEst} in this case as well. Hence \eqref{HkykEst}
holds for any $n\le d$.

\par

By choosing $n=d$ in \eqref{HkykEst} it follows that $a*_{[\theta ]}f$
is uniquely defined and satisfies \eqref{convest1} when
$a\in \ell _0(\zz d)$ and $f\in L^{\mabfp}_{\sigma ,(\omega )}(\rr d)$.
Since $\ell _0$ is dense in $\ell ^{\mabfr}_{\sigma ,(v_\theta )}$, the
result now follows for general $a\in \ell ^{\mabfr}_{\sigma ,
(v_\theta )}(\zz d)$. The proof is complete.
\end{proof}

\par

By choosing $\theta _1=\cdots =\theta _d=1$ and $f$ to be constant
on each open cube $j+(0,1)^d$ in the previous proposition, we get
the following extension of Lemma 2.7 in \cite{GaSa}. The
details are left for the reader.

\par

\begin{cor}
Let $\sigma \in \operatorname {S}_d$, $\omega ,v\in \mascP _E(\rr d)$ be
such that $\omega$ is $v$-moderate, and let
$\mabfp ,\mabfr \in (0,\infty ]^d$ be such that
$$
r_k\le \min _{m\le k}(1,p_m).
$$
Then the map $(a,b)\mapsto a*b$ on $\ell _0(\zz d)$ extends
uniquely to a linear and continuous map from $\ell ^{\mabfr}
_{\sigma ,(v)}(\zz d) \times \ell ^{\mabfp}_{\sigma ,
(\omega )}(\zz d)$ to $\ell ^{\mabfp} _{\sigma ,(\omega )}(\zz d)$.
In particular,
\begin{equation}\label{convest1A}
\nm {a*b}{\ell ^{\mabfp}_{\sigma ,(\omega )}}\le
C\nm {a}{\ell ^{\mabfr}_{\sigma ,(v )}}\nm {b}{\ell ^{\mabfp}
_{\sigma ,(\omega )}},
\end{equation}
for some constant $C$ which is independent of $a\in \ell
^{\mabfr}_{\sigma ,(v)}(\zz d)$ and $b\in \ell ^{\mabfp}
_{\sigma ,(\omega )}(\zz d)$.
\end{cor}

\par

For the link between modulation spaces and sequence spaces we
need to consider a broad family of Wiener spaces.

\par

\begin{defn}\label{DefWienderSpace}
Let $\omega \in \mascP _E(\rr d)$, $\mabfp \in (0,\infty ]^d$,
$q\in [1,\infty]$, $\sigma \in \operatorname S _d$, and let
$\chi$ be the characteristic function of $Q\equiv [0,1]^d$.
Then the \emph{Wiener space} $\sfW ^q(\omega ,\ell
^{\mabfp}_\sigma (\zz d))$ consists of all measurable functions
$f$ on $\rr d$ such that
$$
\nm f{\sfW ^q(\omega ,\ell ^{\mabfp}_\sigma )}\equiv
\nm {b_{f,\omega ,q}}{\ell ^{\mabfp}_\sigma},
$$
is finite, where $b_{f,\omega}$ is the sequence on $\zz d$,
given by
$$
b_{f,\omega}(j)  \equiv \nm {f}{L^q(j+Q)}\omega (j)
= \nm {f\cdot \chi (\cdo -j)}{L^q}\omega (j).
$$
\end{defn}

\par

Especially $\sfW ^\infty (\omega ,\ell ^{\mabfp}_\sigma )$ in Definition
\ref{DefWienderSpace} is important (i.{\,}e. the case
$q=\infty$), and we set
$$
\sfW (\omega ,\ell ^{\mabfp}_\sigma (\zz d)) =
\sfW ^\infty (\omega ,\ell ^{\mabfp}_\sigma (\zz d)).
$$
This space is also called the \emph{coorbit space} of
$L^{\mabfp}_{\sigma}(\rr d)$ with weight $\omega$, and is
sometimes denoted by
$$
\Co (L^{\mabfp}_{\sigma ,(\omega )}(\rr d))\quad \text{or}\quad 
W(L_{\sigma ,(\omega)} ^{\mabfp}) = W(L_{\sigma ,(\omega)}
^{\mabfp}(\rr d)),
$$
in the literature (cf. \cite{Gc2,Rau2}).

\par

We also use the notation
\begin{alignat*}{3}
&\sfW ^q(\ell ^{\mabfp}_\sigma (\zz d))
&\quad &\text{and}\quad
&\sfW (\ell ^{\mabfp}_\sigma (\zz d))
\intertext{instead of}
&\sfW ^q(\omega ,\ell ^{\mabfp}_\sigma (\zz d))
&\quad &\text{and}\quad
&\sfW (\omega ,\ell ^{\mabfp}_\sigma (\zz d)),
\end{alignat*}
respectively, when $\omega =1$.

\par

We have now the following lemma concerning pullbacks of dilations in
Wiener spaces. Here we let $\lfloor  x\rfloor$ denote the integer part of $x$.

\par

\begin{lemma}\label{Wienerdil}
Let $R\ge 1$, $\sigma \in \operatorname {S}_d$, $\omega \in
\mascP _E(\rr d)$, $\theta \in  (0,R]^d$, $\mabfp \in (0,\infty ]^d$,
$q\in (0,\infty]$, and let $f\in {\sfW ^q(\omega ,\ell ^{\mabfp}_\sigma (\zz d))}$.
Then
\begin{align*}
T_\theta ^*f &\in \sfW ^q(T_\theta ^*\omega ,
\ell ^{\mabfp}_\sigma (\zz d)),
\intertext{and}
\nm {T_\theta ^*f}{{\sfW ^q(T_\theta ^*\omega ,\ell ^{\mabfp}_\sigma)}}
&\le
C \left (\prod _{k=1}^d|\theta _k| ^{-1/q} \lfloor 1+|\theta _k|^{-1}\rfloor
^{1/p_{\sigma (k)}}\right ) \nm {f}{{\sfW ^q(\omega ,\ell ^{\mabfp}_\sigma )}},
\end{align*}
for some constant $C$ which only depends on $\omega$ and $R$.
\end{lemma}

\par

\begin{proof}
By considering
$$
f(x_{\sigma (1)},\dots ,x_{\sigma (d)})\quad \text{and}\quad
\omega (x_{\sigma (1)},\dots ,x_{\sigma (d)})
$$
instead of $f(x_1,\dots ,x_d)$ and $\omega (x_1,\dots ,x_d)$,
we reduce ourself to the case when $\sigma$ is the identity map.

\par

Let $Q=[0,1]^d$,
\begin{align*}
\Omega _{n,j} &\equiv  T_\theta (j+Q)\bigcap (n+Q)\subseteq n+Q,
\\[1ex]
I_n &\equiv \sets {j\in \zz d}{\Omega _{n,j}\neq \emptyset} ,
\\[1ex]
M &\equiv \sets {(n_1,j_1,\dots ,n_d,j_d)\in \zz {2d}}{(j_1,\dots ,j_d)\in I_n},
\\[1ex]
\mabfr &\equiv (p_1,p_1,p_2,p_2,\dots ,p_d,p_d)\in (0,\infty ]^{2d}
\intertext{and}
c_1(\theta ) &\equiv \prod _{k=1}^d|\theta _k|^{-1/q}
\end{align*}
Then
\begin{multline*}
\nm {T_\theta ^*f}{{\sfW ^q(T_\theta ^*\omega ,\ell ^{\mabfp})}}
=
\Big \Vert \big \{ \nm {f(T_\theta  \cdo )}{L^q (j+Q)}\omega (T_\theta j)
\big \}_{j\in \zz d} \Big \Vert _{\ell ^{\mabfp} (\zz d)}
\\[1ex]
=
c_1(\theta )
\Big \Vert \big \{ \nm {f}{L^q (T_\theta (j+Q))}\omega (T_\theta j)
\big \}_{j\in \zz d} \Big \Vert _{\ell ^{\mabfp}(\zz d)}
\\[1ex]
c_1(\theta )
\Big \Vert \big \{ \nm {f}{L^q (\Omega _{n,j})}\omega (T_\theta j)
\big \}_{j\in \zz d} \Big \Vert _{\ell ^{\mabfr}(M)}
\\[1ex]
\le
Cc_1(\theta )
\Big \Vert \big \{ \nm {f}{L^q (\Omega _{n,j})}\omega (n)
\big \}_{(n,j)\in M} \Big \Vert _{\ell ^{\mabfr}(M)},
\end{multline*}
where
$$
C= \sup _{x\in \rr d}
\left (\sup _{y\in [-1,R]^d}\omega (x+y)/\omega (x) \right ) <\infty .
$$
Here we use the convention that for any subset $M$ of $\zz d$ and sequence
$a$ on $M$, then $\nm a{\ell ^{\mabfp}(M)}\equiv \nm b{\ell ^{\mabfp}(\zz d)}$,
where $b(j)=a(j)$ when $j\in M$, and $b(j)=0$ otherwise.

\par

Since $ \nm {f}{L^q (\Omega _{n,j})}\le  \nm {f}{L^q (n+Q)}$ and
the number of terms in $I_n$ in direction $k$ is at most
$\lfloor 1+|\theta _k|
^{-1}\rfloor $, we get
\begin{multline*}
\nm {T_\theta ^*f}{{\sfW ^q(T_\theta ^*\omega ,\ell ^{\mabfp})}}
\le
Cc_1(\theta )c_2(\theta ) \Big \Vert \big \{ \nm {f}{L^q (n+Q)}\omega (n)
\big \}_{j\in \zz d} \Big \Vert _{\ell ^{\mabfp}(\zz d)}
\\[1ex]
=
Cc_1(\theta )c_2(\theta )\nm {f} {{\sfW ^q(\omega ,\ell ^{\mabfp})}},
\end{multline*}
where
$$
c_2(\theta )=\prod _{k=1}^d\lfloor 1+|\theta _k|^{-1}\rfloor ^{1/p_k}.
$$
This gives the result.
\end{proof}

\par

\begin{prop}\label{WienerProp}
Let $\sigma \in \operatorname {S}_d$, $\theta \in \mathbf R_*^d$,
$\omega _k\in \mascP _E(\rr d)$, and let $\mabfp _k\in (0,\infty ]^d$,
$q_k\in (0,\infty ]$, $k=1,2,3$, be such that $q_0\ge 1$,
$$
L^{q_1}(\rr d)*L^{q_2}(\rr d)\subseteq L^{q_0}(\rr d)
\quad \text{and}\quad
\ell ^{\mabfp _1}_{\sigma ,(\omega _1)}(\zz d) *
\ell ^{\mabfp _2}_{\sigma ,(\omega _2)}(\zz d) \subseteq
\ell ^{\mabfp _0}_{\sigma ,(\omega _0)}(\zz d),
$$
with continuous embeddings, and
\begin{equation}\label{pqrestriction}
\max (\mabfp _1,q_1)<\infty
\quad \text{or}\quad
\max (\mabfp _2,q_2)<\infty .
\end{equation}
Then the following is true:
\begin{enumerate}
\item The map $(f_1,f_2)\mapsto f_1*f_2$
is continuous from $\sfW ^{q_1}(\omega _1,\ell ^{\mabfp _1}_\sigma
(\zz d)) \times \sfW ^{q_2}(\omega _2,\ell ^{\mabfp _2}_\sigma(\zz d))$ to
$\sfW ^{q_0}(\omega _0,\ell ^{\mabfp _0}_\sigma(\zz d))$, and
$$
\nm {f_1*f_2}{\sfW ^{q_0}(\omega _0,\ell ^{\mabfp _0}_\sigma)}
\lesssim
\nm {f_1}{\sfW ^{q_1}(\omega _1,\ell ^{\mabfp _1}_\sigma)}
\nm {f_2}{\sfW ^{q_2}(\omega _2,\ell ^{\mabfp _2}_\sigma)} \text ;
$$

\vrum

\item The map $(a,f)\mapsto a*_{[\theta ]} f$ is continuous
from $\ell ^{\mabfp _1}_{\sigma ,(T_\theta ^*\omega _1 )}(\zz d) \times
\sfW (\omega _2,\ell ^{\mabfp _2}_\sigma (\zz d))$ to
$\sfW (\omega _0,\ell ^{\mabfp _0}_\sigma (\zz d))$, and
$$
\nm {a*_{[\theta ]} f}{\sfW (\omega _0,\ell ^{\mabfp _0}_\sigma)}
\lesssim
\nm a{\ell ^{\mabfp _1}_{\sigma ,(\omega _1)}}
\nm f{\sfW (\omega _2,\ell ^{\mabfp _2}_\sigma )}
$$
\end{enumerate}
\end{prop}

\par

\begin{proof}
By \eqref{pqrestriction} and density argument, it suffices to prove the
quasi-norm estimates. Furthermore, by a suitable change of variables,
we may assume that $\sigma$ is the identity map.

\par

(1) Let $Q=[0,1]^d$ as usual, and let
$a_k$, $k=0,1,2$, be the sequences on $\zz d$, defined by
$$
a_k(j) \equiv \nm {f_k}{L^{q_k}(j+Q)},\quad j\in \zz d
$$
and $f_0=f_1*f_2$. Then
\begin{multline}\label{a0est}
a_0(j) \le \left ( \int _{j+Q} \left ( \int _{\rr d} |f_1(x-y)f_2(y)|
\, dy \right )^{q_0}\, dx\right )^{1/q_0}
\\[1ex]
=
\left ( \int _{j+Q} \left ( \sum _{j_0\in \zz d}\int _{j_0+Q} |f_1(x-y)f_2(y)|
\, dy \right )^{q_0}\, dx\right )^{1/q_0}
\\[1ex]
\le
\sum _{j_0\in \zz d}
\left ( \int _{j+Q} \left ( \int _{j_0+Q} |f_1(x-y)f_2(y)|
\, dy \right )^{q_0}\, dx\right )^{1/q_0}
\end{multline}

\par

Now, if $x\in j+Q$ and $y\in j_0+Q$, then
$$
x-y\in j-j_0+[-1,1]^d = \bigcup _{n\in \{ 0,1\} ^d} (j-j_0-n+Q).
$$
Hence if $h_k(j,\cdo )=f_k\chi _{j+Q}$, $k=1,2$, then \eqref{a0est} and
Young's inequality give
\begin{multline*}
a_0(j) \le \sum _{n\in \{ 0,1\} ^d}
\sum  _{j_0\in \zz d}\nm {|h_1(j-j_0+n,\cdo )|
* |h_2(j_0,\cdo )|}{L^{q_0}}
\\[1ex]
\le \sum _{n\in \{ 0,1\} ^d}
\sum  _{j_0\in \zz d}\nm {h_1(j-j_0+n,\cdo )}{L^{q_1}}
\nm {h_2(j_0,\cdo )}{L^{q_2}}
\\[1ex]
=\sum _{n\in \{ 0,1\} ^d} (a_1*a_2)(j+n).
\end{multline*}
Here the convolution between $h_1(j_1,x)$ and $h_2(j_2,x)$
should be taken with respect to the $x$-variable only, considering
$j_1$ and $j_2$ as constants.

\par

Now it follows from the assumptions that
\begin{multline*}
\nm {f_1*f_2}{\sfW ^{q_0}(\omega _0,\ell ^{\mabfp _0})}
= \nm {a_0}{\ell ^{\mabfp _0}_{(\omega _0)}}
\\[1ex]
\le
\sum _{n\in \{ 0,1\} ^d} \nm {(a_1*a_2)(\cdo +n)}
{\ell ^{\mabfp _0}_{(\omega _0)}}
\lesssim
\nm {a_1*a_2} {\ell ^{\mabfp _0}_{(\omega _0)}}
\\[1ex]
\lesssim
\nm {a_1}{\ell ^{\mabfp _1}_{(\omega _1)}}
\nm {a_2}{\ell ^{\mabfp _2}_{(\omega _2)}}
=
\nm {f_1}{\sfW ^{q_1}(\omega _1,\ell ^{\mabfp _1})}
\nm {f_2}{\sfW ^{q_2}(\omega _2,\ell ^{\mabfp _2})},
\end{multline*}
and the result follows in this case. Here the second inequality
follows from the fact that $\omega _0$ is $v$-moderate for
some $v$. This gives (1).

\par

(2) Since $\omega$ is $v$-moderate we get
$$
|a*_{[\theta ]} f|\cdot \omega  \lesssim |a\cdot (\omega \circ T_\theta )|
*_{[\theta ]} |f\cdot v| ,
$$
which reduce the situation to the case when $\omega =v=1$.
Furthermore, since
\begin{align*}
\nm {a*_{[\theta ]} f}{\sfW (\ell ^{\mabfp })}
&\le
\nm {|a|*_{[\theta ]} g}{\sfW (\ell ^{\mabfp })},
\quad
\nm a{\ell ^{\mabfp}} = \nm {\, |a|\, }{\ell ^{\mabfp}}
\intertext{and}
\nm f{W(L^\mabfp )} &\asymp \nm g{\sfW (\ell ^\mabfp )}
\intertext{when}
T_\theta ^*g &= \sum _{j\in \zz d}\nm f{L^\infty (j+Q)}\chi _{j+Q} ,
\end{align*}
we may assume that $a\ge 0$ and $(T_\theta ^*f)(x) =b(j)\ge 0$
when $x\in j+Q$.

\par

Let $\theta _0=(1,\dots ,1)$. By Lemma \ref{Wienerdil} we get
\begin{multline*}
\nm {a*_{[\theta ]} f}{\sfW (\ell ^{\mabfp _0})}
\asymp
\nm {T_\theta ^*(a*_{[\theta ]} f)}{\sfW (\ell ^{\mabfp _0})}
=
\nm {a*_{[\theta _0]} (T_\theta ^*f)}{\sfW (\ell ^{\mabfp _0})}
\\[1ex]
=
\nm {a*b}{\ell ^{\mabfp _0}}
\lesssim
\nm a{\ell ^{\mabfp _1}}\nm b{\ell ^{\mabfp _2}}
\asymp 
\nm a{\ell ^{\mabfp _1}}\nm f{\sfW (\ell ^{\mabfp _2})},
\end{multline*}
and the result follows.
\end{proof}

\par

\section{Time-frequency representation of modulation spaces}
\label{sec3}

\par

In this section we extend the Gabor analysis for modulation spaces
of the form $M^{p,q}_{(\omega)}(\rr {d})$ with $p,q\in (0,\infty ]$ and
$\omega \in \mascP(\rr {2d})$ in \cite{GaSa}, to spaces of the form
$M^{\mabfp} _{\sigma ,(\omega)}(\rr {d})$ with $\sigma \in
\operatorname S_{2d}$, $\mabfp \in (0,\infty ]^{2d}$ and
$\omega \in \mascP _E(\rr {2d})$. Especially we deduce invariance
properties for $M^{\mabfp} _{\sigma ,(\omega)}(\rr {d})$ concerning
the choice of the window function $\phi$ in \eqref{modnorm2}, and that
the results on Gabor expansions in \cite{GaSa,Gc2} also hold in this
more general situation. As a consequence we deduce that
$M^\mabfp _{(\omega)}$ increases with $\mabfp$.

\par

We have now the following proposition.

\par

\begin{prop}\label{propwindowindep}
Let $\mabfp \in (0,\infty ]^{2d}$, $r=\min (1,\mabfp )$,
$\omega ,v\in \mascP _E(\rr {2d})$ be such that $\omega$ is
$v$-moderate, and let $\Theta _\rho v$ be the same as in Remark
\ref{RemThmS}. Also let $\sigma \in \operatorname
S_{2d}$, $\phi _1,\phi _2\in M^1_{\sigma ,(\Theta _\rho v)}
(\rr d)\setminus 0$, and let $f\in \Sigma _1'(\rr d)$. Then
$$
\nm {V_{\phi _1}f}{L^{\mabfp}_{\sigma ,(\omega )}} \le C
\nm {V_{\phi _2}f}{L^{\mabfp}_{\sigma ,(\omega )}},
$$
for some constant $C$ which is independent of $f\in
\Sigma _1'(\rr d)$. In particular, the modulation space
$M^{\mabfp}_{\sigma ,(\omega )}(\rr d)$ is independent of the
choice of $\phi \in M^1_{\sigma ,(\Theta _sv)}(\rr d)\setminus 0$ in
\eqref{modnorm2}, and different choices of $\phi$ give rise to
equivalent norms.
\end{prop}

\par

The proof follows by similar arguments as the proof of
Theorem 3.1 in \cite{GaSa}. In order to be self-contained we
here present a proof.
For the proof we need the following lemma on point estimates
for short-time Fourier transforms with Gaussian windows. The result
is a slight extension of Lemma 2.3 in \cite{GaSa}. Here and in
what follows we let $B_r(x_0)$ be the open ball in $\rr d$ with center at
$x_0\in \rr d$ and radius $r>0$.

\par

\begin{lemma}\label{SubharmonicEst}
Let $p\in (0,\infty ]$, $r>0$, $(x_0,\xi _0)\in \rr {2d}$ be fixed,
and let $\phi \in \maclS _{1/2}(\rr d)$ be a Gaussian. Then
$$
|V_\phi f(x_0,\xi _0)| \le C\nm {V_\phi f}{L^p(B_r(x_0,\xi _0))},\quad
f\in \maclS _{1/2}'(\rr d),
$$
where the constant $C$ is independent of $(x_0,\xi _0)$ and $f$.
\end{lemma}

\par

When proving Lemma \ref{SubharmonicEst} we may first reduce ourself
to the case that the Gaussian $\phi$ should be centered at origin, by
straight-forward arguments involving pullbacks with translations.
The result then follows by using the same arguments as in
\cite[Lemma 2.3.]{GaSa} and its proof, based on the fact that
$$
z\mapsto F_w(z) \equiv e^{c_1|z|^2+c_2(z,w)+c_3|w|^3}V_\phi f(x,\xi ),
\quad z=x+i\xi
$$
is an entire function for one choice of the constant
$c_1$ (depending on $\phi$).

\par

\begin{rem}
We note that Lemma 2.3 and its proof in \cite{GaSa} contains a
mistake, which is not
important in the applications. In fact, when using the mean-value
inequality for subharmonic functions in the proof, a factor of
the volume for the ball which corresponds to $B_r(x_0,\xi _0)$ in Lemma
\ref{SubharmonicEst} is missing. This leads to that stated invariance
properties of constants in several results in \cite{GaSa} are
more dependent of the involved parameters than what are stated.
\end{rem}

\par

\begin{proof}[Proof of Proposition \ref{propwindowindep}]
Let $v_0=\Theta _\rho v$, and let
$$
\Lambda =\ep  \zz {2d}=\sets {{\ep x_j,\ep \xi _k}}{j,k\in J},
$$
where $J$ is an index set and $\ep >0$ is chosen small enough such
that $\{ e^{i\scal \cdo {\xi _k}}\phi _1(\cdo -x_j) \} _{j,k\in
J}$ is a Gabor frame. Since
$\phi _2\in M^1_{(v_0)}$, it follows that its dual window $\psi$
belongs to $M^1_{(v_0)}$, in view of Proposition \ref{ThmS}. By
Proposition \ref{p1.4B} (3) we have
$$
\phi _1 =\sum _{j,k\in J} (V_\psi \phi _1)(x_j,\xi _k)
e^{i\scal \cdo {\xi _k}}\phi _2(\cdo -x_j),
$$
with unconditional convergence in $M^1_{(v_0)}$. This gives,
\begin{multline*}
|V_{\phi _1}f(x,\xi )| = (2\pi )^{-d/2}|(f,e^{i\scal \cdo {\xi}}
\phi _1(\cdo -x))|
\\[1ex]
\le (2\pi )^{-d/2}\sum _{j,k\in J}
|(V_\psi \phi _1)(x_j,\xi _k)|  |(f,e^{i\scal \cdo {\xi +\xi _k}}
\phi _2(\cdo -x-x_j))|
\\[1ex]
=  (2\pi )^{-d/2}\sum _{j,k\in J}
|(V_\psi \phi _1)(x_j,\xi _k)|  |V_{\phi _2}(x+x_j,\xi +\xi _j)|
=
(|b| *_{[\theta ]} |V_{\phi _2}f|)(x,\xi  ),
\end{multline*}
where $b(x_j,\xi _k) = |(V_\psi \phi _1)(-\ep x_j,-\ep \xi _k)|$, and
$\theta _j=\ep$, $j=1,\dots ,2d$.

\par

By Proposition \ref{PropSemiContConvEst}  and Lemma \ref{Wienerdil}
we get  with $r=\min \mabfp$,
\begin{multline*}
\nm {V_{\phi _1}f}{L^{\mabfp}_{\sigma ,(\omega )}}
\lesssim
\nm b{\ell ^r_{\sigma ,(v)}}
\nm {V_{\phi _2}f}{L^{\mabfp}_{(\omega )}}
\\[1ex]
\lesssim
\nm b{\ell ^1_{(v_0)}}\nm {V_{\phi _2}f}{L^{\mabfp}_{\sigma ,(\omega )}}
\asymp
\nm {V_\psi \phi _1}{L^1_{(v_0)}}
\nm {V_{\phi _2}f}{L^{\mabfp}_{\sigma ,(\omega )}}.
\end{multline*}
Here we have used the fact that $L^{\mabfp}_{\sigma ,(\omega )}(\rr {2d})
=L^{\mabfp}_{(\omega )}(\rr {2d})$ when $p_1=\cdots =p_{2d}$. Since
$\nm {V_\psi \phi _1}{L^1_{(v_0)}} \asymp \nm {\phi _1}{M^1_{(v_0)}}<\infty$
by Proposition 12.1.2 in \cite{Gc2}, the result follows.
\end{proof}

\par

We have now the following result related to \cite[Theorem 3.3]{GaSa}.

\par

\begin{prop}\label{WienerEquiv}
Let $\mabfp \in (0,\infty ]^{2d}$, $\omega ,v\in \mascP _E(\rr {2d})$
be such that $\omega$ is $v$-moderate, $\Theta _\rho v$ be the
same as in Proposition \ref{propwindowindep}, $\phi _1,\phi _2\in
M^1_{(\Theta _\rho v)} (\rr d)\setminus 0$, $\sigma
\in \operatorname {S}_{2d}$ and let $\omega \in \mascP _E(\rr {2d})$.
Then $V_{\phi _1} f \in L^\mabfp _{\sigma ,(\omega )}(\rr {2d})$,
if and only if $V_{\phi _2} f \in \sfW (\omega ,\ell ^\mabfp _{\sigma}
(\zz {2d}))$, and
$$
\nm {V_{\phi _1} f}{L^\mabfp _{\sigma ,(\omega )}}
\asymp 
\nm {V_{\phi _2} f}{\sfW (\omega ,\ell ^\mabfp _{\sigma})},\quad f\in
\maclS '_{1/2}(\rr d).
$$
\end{prop}

\par

For the proof we note that
for every measurable function $F$ on $\rr {2d}$ we have
\begin{equation}\label{Wien1LebEst}
\nm {F}{\sfW ^1 (v,\ell ^r_\sigma )}\lesssim
\nm F{L^1_{\sigma ,(\Theta _\rho v)}},
\end{equation}
which follows by an application of H{\"o}lder's inequality.

\par

\begin{proof}
By the definitions it follows that
$$
\nm {V_\phi f}{L^\mabfp _{\sigma ,(\omega )}}
\lesssim
\nm {V_\phi f}{\sfW (\omega, \ell ^\mabfp _{\sigma})},
$$
when $\phi \in \maclS _{1/2}$.

\par

When proving the reversed inequality we start by considering the case
when $\phi _1=\phi _2=\phi$ is a Gaussian. First we need to introduce
some notations. We set
\begin{align*}
X &= (X_1,\dots ,X_{2d}) = (x_1,\dots ,x_d,\xi _1,\dots ,\xi _d)
\\[1ex]
Y &= (Y_1,\dots ,Y_{2d}) = (X_{\sigma (1)},\dots , X_{\sigma (2d)}),
\\[1ex]
r &= \min \mabfp
\quad \text{and}\quad
F(Y) = |V_\phi f(X)|\omega (X)
\end{align*}
For every $k\in \{ 0,\dots ,2d \}$ we also set
\begin{alignat}{2}
\mabfq _k &= (p_1,\dots ,p_{k}),&\quad
\mabfr _k &= (p_{k+1},\dots ,p_{2d}),\notag
\\[1ex]
t_k &= (Y_{k+1},\dots ,Y_{2d}), & \quad
Q_k &= [-2,2]^k,\notag
\intertext{and}
b_k(l) &\equiv \left (
\int _{l+Q_{2d-k}}\nm {F(\cdo ,t_k)}{L^{\mabfq _k}}^r\, dt_k
\right )^{1/r},
&\quad 
l &\in \zz {2d-k},\ k<2d\notag
\\[1ex]
b_{2d} & \equiv \nm {F}{L^{\mabfq _{2d}}} = \nm {V_\phi f}
{L^\mabfp _{\sigma ,(\omega )}}. & & \label{Fk2dest}
\end{alignat}

\par

We claim that for every $k\in \{ 1,\dots , 2d\}$, the inequality
$$
\nm {V_\phi f}{\sfW (\omega ,\ell ^{\mabfp}_{\sigma} )}
\lesssim \nm {b_k}{\ell ^{\mabfr _k}}
$$
holds.

\par

In fact, for $k=1$, the result follows from Lemmas \ref{Wienerdil} and
\ref{SubharmonicEst}, H{\"o}lder's inequality and the fact that $\omega$
is moderate.

\par

Assume that the result is true for $k\in \{ 1,\dots , 2d-1\}$, and prove the
result for $k+1$. For notational convenience we only prove the statement in the
case $p_0=p_{k+1}<\infty$. The case $p_{k+1}<\infty$ follows by similar arguments
and are left for the reader.

\par

Let $t=t_{k+1}$ and
$$
c_{k}(l)
=
\left ( \sum _{j\in \mathbf Z} \left ( \int _{l+Q_{2d-k-1}} \int _{j+Q_1}
\nm {F(\cdo ,z,t)}{L^{\mabfq _k}}^r\, dzdt
\right ) ^{p_0/r}\right ) ^{1/p_0}.
$$
Then 
$$
\nm {b_k(\cdo ,l)}{\ell ^{p_k+1}(\mathbf Z)} =c_{k}(l),\qquad l\in \zz {2d-k-1},
$$
giving that
\begin{equation}\label{STFTNormEst}
\nm {V_\phi f}{\sfW (\omega ,\ell ^{\mabfp}_{\sigma} )}
\lesssim
\nm {b_k}{\ell ^{\mabfr _k}(\zz {2d-k})} = \nm {c_{k}}{\ell
^{\mabfr _{k+1}} (\zz {2d-k-1})}.
\end{equation}

\par

Since $p_0\ge r$, Minkowski's and H{\"o}lder's inequalities give
\begin{multline*}
c_{k}(l)
= 
\left ( \sum _{j\in \mathbf Z} \left ( \int _{l+Q_{2d-k-1}} \int _{-2}^2
\nm {F(\cdo ,z+j,t)}{L^{\mabfq _k}}^r\, dzdt
\right ) ^{p_0/r}\right ) ^{1/p_0}
\\[1ex]
\le
\left ( \sum _{j\in \mathbf Z} \left ( \int _{l+Q_{2d-k-1}}
\left ( \int _{-2}^2 \nm {F(\cdo ,z+j,t)}{L^{\mabfq _k}}^{p_0}
\, dz\right ) ^{r/p_0}dt
\right ) ^{p_0/r}\right ) ^{1/p_0}
\\[1ex]
\le
\left ( \int _{l+Q_{2d-k-1}} \left ( \sum _{j\in \mathbf Z}
 \int _{-2}^2 \nm {F(\cdo ,z+j,t)}{L^{\mabfq _k}}^{p_0}
\, dz \right ) ^{r/p_0}dt \right ) ^{1/r}
\\[1ex]
\asymp
\left ( \int _{l+Q_{2d-k-1}} \left ( 
\int _{\mathbf R} \nm {F(\cdo ,z,t)}{L^{\mabfq _k}}^{p_0}
\, dz\right ) ^{r/p_0}dt
\right ) ^{1/r}
=
b_{k+1}(l),
\end{multline*}
and the induction step follows from these estimates,
\eqref{Fk2dest} and \eqref{STFTNormEst}. This gives the result
when $\phi$ is a Gaussian.

\par

Next assume that $\phi \in M^1_{(\Theta _\rho v)}\setminus 0$ is arbitrary,
and let $\phi _0$ be a fixed Gaussian. Then
$$
\nm {V_\phi f}{L^\mabfp _{\sigma ,(\omega )}}
\asymp
\nm {V_{\phi _0} f}{L^\mabfp _{\sigma ,(\omega )}},
$$
by Proposition \ref{propwindowindep}, and the result
follows if we prove 
\begin{equation}\label{Estagain}
\nm {V_{\phi} f}{\sfW (\omega ,\ell ^\mabfp _{\sigma})}
\lesssim
\nm {V_{\phi _0} f}{\sfW (\omega ,\ell ^\mabfp _{\sigma})}.
\end{equation}

\par

We have
$$
|V_\phi f|\lesssim |V_{\phi_0} f|*|V_\phi \phi _0|,
$$
(cf. \cite[Chapter 11]{Gc2}). An application of
Proposition \ref{WienerProp} gives
\begin{equation*}
\nm {V_\phi f}{\sfW (\omega ,\ell ^{\mabfp}_{\sigma})}
\lesssim \nm {V_\phi \phi _0}{\sfW ^1 (v,\ell ^r_\sigma )}
\nm {V_{\phi_0} f}{\sfW (\omega ,\ell ^{\mabfp}_{\sigma})}
\lesssim
\nm {V_{\phi_0} f}{L^{\mabfp}_{\sigma ,(\omega )}}.
\end{equation*}
Here the last inequality follows from \eqref{Wien1LebEst}.
This gives \eqref{Estagain}, and the result follows.
\end{proof}

\par

The next result is an immediate consequence of the previous proposition and
the fact that $\ell ^{\mabfp}_{(\omega )}$ is increasing with respect to
$\mabfp$, giving that $\sfW (\omega ,\ell ^\mabfp _{\sigma})$ increases
with $\mabfp$, when $\omega \in \mascP _E$.

\par

\begin{prop}\label{ModSpIncrease}
Let $\sigma \in \operatorname {S}_{2d}$, $\mabfp _1,
\mabfp _2\in (0,\infty ]^{2d}$
and $\omega _1,\omega _2 \in \mascP _E(\rr {2d})$ be such that
be such that $\mabfp _1\le \mabfp _2$ and $\omega _2\lesssim
\omega _1$. Then
$$
M^{\mabfp _1}_{\sigma ,(\omega _1)}(\rr d)\subseteq M^{\mabfp _2}
_{\sigma ,(\omega _2)}(\rr d),
$$
and
$$
\nm f{M^{\mabfp _2}_{\sigma ,(\omega _2)}}
\lesssim
\nm f{M^{\mabfp _1}_{\sigma ,(\omega _1)}},
\qquad f\in \Sigma _1'(\rr d).
$$
\end{prop}

\par

Next we extend the Gabor analysis in \cite{GaSa} to modulation spaces
of the form $M^\mabfp _{\sigma ,(\omega )}$, with Lebesgue exponents
and weights as before. The first
two results show that the analysis and synthesis operators posses the
requested continuity properties.

\par

\begin{prop}\label{AnSynthOpProp}
Let $\Lambda =T_\theta \zz {2d}$ for some $\theta\in \mathbf R_*^{2d}$,
$\mabfp \in (0,\infty ]^{2d}$, $0<r\le \min
(1,\mabfp )$, and let $\omega ,v\in \mascP _E(\rr {2d})$
be such that $\omega$ is $v$-moderate.
Also let $\phi ,\psi \in M^r_{(v)}(\rr d)$, and let $C_\phi$ and $D_\psi$
be as in Definition \ref{DefAnSynGabOps}. Then the following
is true:
\begin{enumerate}
\item $C_\phi$ is uniquely extendable to 
continuous map from $M^\mabfp _{\sigma ,(\omega)}(\rr d)$
to $\ell ^\mabfp _{\sigma ,(\omega)}(\Lambda )$;

\vrum

\item $D_\psi$ is uniquely extendable to 
continuous map from $\ell ^\mabfp _{\sigma ,(\omega)}(\Lambda )$
to $M^\mabfp _{\sigma ,(\omega)}(\rr d)$.
\end{enumerate}

\par

Furthermore, if $\max \mabfp <\infty$, $f\in
M^\mabfp _{\sigma ,(\omega)}(\rr d)$ and $c\in \ell
^\mabfp _{\sigma ,(\omega)}(\Lambda )$, then $C_\phi f$
and $D_\psi c$ converge unconditionally and in norms. If
instead $\max \mabfp
=\infty$, then $C_\phi f$ and $D_\psi c$ converge in the
weak$^*$ topology in $\ell ^\infty _{(\omega)}(\Lambda )$
and $M^\infty _{(\omega)}(\rr d)$, respectively.
\end{prop}

\par

\begin{proof}
We shall mainly follow the proofs of Theorems 3.5 and 3.6 in \cite{GaSa}.
It suffices to prove the desired norm estimates
\begin{equation}\label{DesNormEsts}
\nm {C_\phi f}{\ell ^\mabfp _{\sigma ,(\omega)}(\Lambda )}
\lesssim
\nm f{M^\mabfp _{\sigma ,(\omega)}}
\quad \text{and}\quad
\nm {D_\psi c}{M^\mabfp _{\sigma ,(\omega)}}
\lesssim
\nm {c}{\ell ^\mabfp _{\sigma ,(\omega)}(\Lambda )},
\end{equation}
when $f\in \Sigma _1(\rr d)$ and $c\in \ell _0(\Lambda )$.

\par

In fact, if $\max \mabfp <\infty$, then the result
follows from \eqref{DesNormEsts} and the fact that
$\Sigma _1$ and $\ell _0$ are dense in 
$M^\mabfp
_{\sigma ,(\omega)}$ and $\ell ^\mabfp _{\sigma ,(\omega)}$,
respectively. If instead
$\max \mabfp =\infty$, then the result follows
from the facts that both $M^\mabfp
_{\sigma ,(\omega)}$ and $\ell ^\mabfp _{\sigma ,(\omega)}$
increase with $\mabfp$, and that $\Sigma _1$ and
$\ell _0$ are dense in $M^\infty _{(\omega)}$
and $\ell ^\infty _{(\omega)}$, respectively, with respect to the
weak$^*$-topologies.

\par

In order to prove the first inequality in \eqref{DesNormEsts},
let $\Lambda =\{ (x_j,\xi _k) \} _{j,k\in J}$ as before.
Then
$$
C_\phi f = \{ V_\phi f(x_j,\xi _k) \} _{j,k \in J},
$$
and Propositions \ref{propwindowindep} and \ref{WienerEquiv} gives
\begin{equation*}
\nm {C_\phi f}{\ell ^\mabfp _{\sigma ,(\omega)}(\Lambda )}
\lesssim \nm {V_\phi f}{\sfW (\omega ,\ell ^\mabfp _{\sigma})}
\asymp
\nm {V_\phi f}{L^\mabfp _{\sigma ,(\omega )}}
\asymp \nm f{M^\mabfp _{\sigma ,(\omega)}}.
\end{equation*}
This gives the first estimate in \eqref{DesNormEsts}.

\par

For the second estimate in \eqref{DesNormEsts},
let $\phi _0\in \Sigma _1$ be fixed. Then
\begin{multline*}
|V_{\phi _0}(D_\psi c)(x,\xi )|
=
\left | \sum _{(j,k\in J}c(x_j,\xi _k)V_{\phi _0}
\left ( e^{i\scal {\cdo}{\xi _k}} \psi (\cdo -x_j)\right )(x,\xi ) \right |
\\[1ex]
\le (b *_{[\theta ]}|V_{\phi _0}\psi |)(x,\xi ),
\end{multline*}
where $b \equiv T^*_\theta |c|$ is a sequence on $\zz {2d}$.
Hence, by letting $\mabfp _0=\mabfp _1=\mabfp$ and
$\mabfp _2=\mabfr$ in (2) in Proposition \ref{WienerProp},
Propositions \ref{PropSemiContConvEst} gives
\begin{multline*}
\nm {D_\psi c}{M^\mabfp _{\sigma ,(\omega)}}
\asymp
\nm {b *_{[\theta ]}|V_{\phi _0}\psi |}
{L^\mabfp _{\sigma ,(\omega)}}
\lesssim
\nm {b *_{[\theta ]}|V_{\phi _0}\psi |}
{\sfW (\omega ,\ell ^\mabfp _{\sigma})}
\\[1ex]
\lesssim
\nm b{\ell ^\mabfp _{\sigma ,(\omega \circ T_\theta )}(\zz {2d})}
\nm {V_{\phi _0}\psi}{\sfW (v,\ell ^\mabfr _{\sigma})}
\asymp
\nm c{\ell ^\mabfp _{\sigma ,(\omega )}(\Lambda )}
\nm {V_{\phi _0}\psi}{L^\mabfr _{\sigma ,(v)}}
\\[1ex]
\asymp
\nm c{\ell ^\mabfp _{\sigma ,(\omega )}(\Lambda )}
\nm {\psi}{M^\mabfr _{\sigma ,(v)}},
\end{multline*}
and the result follows.
\end{proof}

\par

As a consequence of the last proposition we get the following.

\par

\begin{thm}\label{ModGabThm}
Let $\Lambda = T_\theta \zz {2d}=\{ (x_j,\xi _k )\} _{j,k\in J}$, where $\theta \in
\mathbf R_*^{2d}$, $\mabfp ,\mabfr \in
(0,\infty ]^{2d}$, $\sigma \in \operatorname
{S}_{2d}$, and let $\omega ,v\in \mascP _E(\rr {2d})$ be the same as in
Proposition \ref{AnSynthOpProp}. Also let $\phi ,\psi \in
M^{\mabfr}_{(v)}(\rr d)$ be such that
\begin{equation}\label{DualFrames}
\{ e^{i\scal {\cdo }{\xi _k}}\phi (\cdo -x_j) \} _{j,k\in J}
\quad \text{and}\quad
\{ e^{i\scal {\cdo }{\xi _k}}\psi (\cdo -x_j) \} _{j,k\in J}
\end{equation}
are dual frames to each others. Then the following is true:
\begin{enumerate}
\item The operators $S_{\phi ,\psi} \equiv D_\psi \circ C_\phi$ and
$S_{\psi ,\phi} \equiv D_\phi \circ C_\psi$ are both the identity map
on $M^\mabfp _{\sigma ,(\omega )}(\rr d)$, and
\begin{align}
f &= \sum _{j,k\in J} (V_\phi f)(x_j,\xi _k)
e^{i\scal {\cdo }{\xi _k}}\psi (\cdo -x_j)\notag
\\[1ex]
&=
\sum _{j,k\in \zz d} (V_\psi f)(x_j,\xi _k)
e^{i\scal {\cdo }{\xi _k}}\phi (\cdo -x_j),\label{GabExpForm}
\end{align}
with unconditional norm-convergence in $M^\mabfp _{\sigma ,(\omega )}$
when $\max \mabfp <\infty$, and with convergence in
$M^\infty _{(\omega)}$ with respect to the weak$^*$ topology otherwise;

\vrum

\item
$
\displaystyle{\nm f{M^\mabfp _{\sigma ,(\omega )}}
\asymp
\nm {(V_\phi f)\circ T_{\theta }}
{\ell ^\mabfp _{\sigma ,(\omega \circ T_{\theta  })} }
\asymp
\nm {(V_\psi f)\circ T_{\theta }}
{\ell ^\mabfp _{\sigma ,(\omega \circ T_{\theta  })}} }.
$
\end{enumerate}
\end{thm}

\par

\begin{proof}
By Corollary 12.2.6 in \cite{Gc2}, the result follows in the case
$M^{\mabfp}_{\sigma (\omega )} = M^\infty _{(\omega )}$.
Since $M^{\mabfp}_{\sigma (\omega )}$ increases with $\mabfp$,
the identity \eqref{GabExpForm} holds for any $f\in
M^{\mabfp}_{\sigma (\omega )}$. The result now follows from
Proposition \ref{AnSynthOpProp} and the facts that $\ell _0$
and $\Sigma _1$ are dense in $\ell ^{\mabfp}_{\sigma ,(\omega )}$
and $L^{\mabfp}_{\sigma ,(\omega )}$, respectively, when
$\max \mabfp <\infty$.
\end{proof}

\par

We shall end the section by applying the latter results to deduce
invariance properties of compactly supported elements in
$M^{p,q}_{(\omega )}$ and in $W^{p,q}_{(\omega )}$. The space
$W^{p,q}_{(\omega )}(\rr d)$, with $p,q\in (0,\infty ]$ and $\omega
\in \mascP _E(\rr {2d})$, is the Wiener amalgam related space, defined
as the set of all $f\in \maclS _{1/2}'(\rr d)$ such that
$$
\nm f{W^{p,q}_{(\omega )}}\equiv \nm {V_\phi f \cdot \omega}{L^{p,q}_*}
$$
is finite. Here $L^{p,q}_*(\rr {2d})$ is the set of all measurable
$F$ on $\rr {2d}$ such that
$$
\nm F{L^{p,q}_*}\equiv \nm {f_0}{L^p}<\infty,\quad
\text{where} \quad
f_0(x) = \nm {F(x,\cdo )}{L^q}.
$$
Evidently, $M^{\mabfp}_{\sigma ,(\omega )}=W^{p,q}_{(\omega )}$ for suitable
$\mabfp \in (0,\infty ]^{2d}$ and $\sigma \in \operatorname S_{2d}$.

\par

As a consequence of  Remark 4.6 in \cite{RSTT} and its arguments,
it follows that
\begin{equation}\label{ModCompSuppInv}
\begin{aligned}
M^{p_1,q}_{(\omega )}(\rr d) \bigcap \maclE '_t(\rr d)
&= W^{p_2,q}_{(\omega )}(\rr d) \bigcap \maclE '_t(\rr d)
\\[1ex]
&= \mathscr F L^{q}_{(\omega )}(\rr d) \bigcap \maclE '_t(\rr d)
\end{aligned}
\end{equation}
when $p_1,p_2,q\in [1,\infty ]$ and $t>1$. Here $\maclE '_t(\rr d)$
is the set of compactly supported elements in $\maclS _t'(\rr d)$,
and for any $q\in (0,\infty ]$ and $\omega \in
\mascP _E(\rr {2d})$, the set $\mascF L^q_{(\omega )}(\rr d)$
consists of all $f\in \Sigma _1'(\rr d)$ such that $\widehat f$ is
measurable and belongs to $L^q_{(\omega )}(\rr d)$. We set
$$
\nm f{\mascF L^q_{(\omega )}}
= \nm f{\mascF L^q_{x,(\omega )}}
\equiv
\nm {\widehat f \cdot \omega (x,\cdo )}{L^q}.
$$
Note here that if $x\in \rr d$ is fixed, then
$$
\nm {\widehat f \cdot \omega (x,\cdo )}{L^q}
\asymp
\nm {\widehat f \cdot \omega (0,\cdo )}{L^q},
$$
since $\omega$ is $v$-moderate for some $v$. Consequently,
the condition $\nm f{\mascF L^q_{x,(\omega )}}<\infty$
is independent of $x\in \rr d$, though the norm
$\nm f{\mascF L^q_{x,(\omega )}}$ might depend on $x$.

\par

We have now the following extension of \cite[Remark 4.6]{RSTT}.

\par

\begin{prop}\label{CompSuppInv}
Let $\omega \in \mascP _E(\rr {2d})$, $p,q\in (0,\infty ]$ and $t>1$.
Then \eqref{ModCompSuppInv} holds. In particular,
$$
M^{p,q}_{(\omega )}(\rr d) \bigcap \maclE '_t(\rr d)
\quad \text{and}\quad
W^{p,q}_{(\omega )}(\rr d) \bigcap \maclE '_t(\rr d)
$$
are independent of $p$.
\end{prop}

\par

We need the following lemma for the proof. Here the first part
follows from \cite[Proposition 4.2]{CPRT05}.

\par

\begin{lemma}\label{LemmaSTFTCompSupp}
Let $1<s<t$ and let $f\in \maclE _t'(\rr d)$.
Then the following is true:
\begin{enumerate}
\item if $\phi \in \maclS _s(\rr d)$, then
$$
|V_\phi f(x,\xi )| \lesssim e^{-h|x|^{1/t}}e^{\ep |\xi| ^{1/t}},
$$
for every $h>0$ and $\ep >0$.

\vrum

\item $|\widehat f(\xi )|\lesssim e^{\ep |\xi| ^{1/t}}$, for every $\ep >0$.
\end{enumerate}
\end{lemma}

\par

\begin{proof}
The first part follows from \cite[Proposition 4.2]{CPRT05}.

\par

By choosing $\phi$ in (1) such that $\int \phi \, dx =1$, we get
$$
|\widehat f(\xi )| = \left | \int  V_\phi f(x,\xi )\, dx \right |
\le \nm {V_\phi f(\cdo ,\xi )}{L^1}\lesssim e^{\ep |\xi| ^{1/t}},
$$
where the last estimate follows from (1).
\end{proof}

\par

\begin{proof}[Proof of Proposition \ref{CompSuppInv}]
We use the same notations as in the proofs of Proposition
\ref{AnSynthOpProp} and Theorem \ref{ModGabThm}.
First assume that $\omega \ge 1/v_s$ for some $s$
satisfying $1<s<t$, where
$$
v_s(x,\xi ) = e^{r(|x|^{1/s} + |\xi |^{1/s})}
$$
and $r>0$ is fixed. Let $\phi ,\psi \in \cap _{r>0}M^r_{(v_s)}(\rr d)$
and $\{(x_j,\xi _j)\}$ be such that \eqref{DualFrames} are dual
Gabor frames, and such that $\phi$ has compact support.
Such frames exists in view of Proposition \ref{ThmS}, and the fact
that $\maclD _{s_0}$ is non-trivial and contained in $\maclS _{s_0}
\subseteq M^1_{(v_s)}$ when $1<s_0<s$. Here the latter inclusion
follows from \cite[Theorem 3.9]{Toft8}. For conveniency we
assume that $0\in J$ and $x_0=\xi _0=0$.

\par

By Theorem \ref{ModGabThm} it follows that any
$f\in \maclS _t'\subseteq M^{p,q}_{(1/v_s)}$ possess the
expansions \eqref{GabExpForm}, and that
\begin{equation}\label{EstCoeff}
\nm f{M^{p,q}_{(\omega )}}\asymp \nm {c}{\ell ^{p,q}_{(\omega )}}
=\nm {c}{\ell ^{p,q}_{(\omega )}(\Lambda )}, \qquad c=\{ c(j,k) \} _{j,k\in J},
\end{equation}
where $c(j,k)=(V_\phi f)(x_j,\xi _k)$. Furthermore, if 
$\ell ^{p,q}_{*,(\omega )}$ is the set of all $b=\{ b(j,k)\} _{j,k\in J}$
such that
$$
\nm {b_0}{\ell ^p}<\infty ,\qquad b_0(j) =
\nm {b(j,\cdo )\omega (j,\cdo )}{\ell ^q},
$$
then
$$
\nm f{W^{p,q}_{(\omega )}}\asymp
\nm {c}{\ell ^{p,q}_{*,(\omega )}}
$$

\par

Now assume that in addition $f\in \maclE _t '(\Omega )$, for
some bounded and open set $\Omega \subseteq \rr d$.
Since both $f$ and $\phi$ has compact supports, it follows
that there is a finite set $J_0\subseteq J$ such that $c_{j,k}=0$
when $j\in J\setminus J_0$. This implies that
$$
\nm c{\ell ^{p_1,q}_{(\omega )}} \asymp
\nm c{\ell ^{p_2,q}_{*,(\omega )}},
$$
for every $p_1,p_2\in (0,\infty ]$, and the first equality in
\eqref{ModCompSuppInv} follows in this case.

\par

Next let $\omega \in \mascP _E(\rr {2d})$ be general. Since
$\maclE _t'(\rr d)\subseteq M^{p,q}_{(1/v_s)}$, it follows that
\begin{equation}\label{MCapComp}
M^{p,q}_{(\omega )}\bigcap \maclE _t'
= \left ( M^{p,q}_{(\omega )}\bigcap M^{p,q}_{(1/v_s)} \right )
\bigcap \maclE _t' 
= M^{p,q}_{(\omega +1/v_s)}\bigcap \maclE _t' .
\end{equation}
In the same way it follows that
$$
W^{p,q}_{(\omega )}\bigcap \maclE _t'
= W^{p,q}_{(\omega +1/v_s)}\bigcap \maclE _t' .
$$
The first equality in \eqref{ModCompSuppInv} now follows from
these identities, the first part of the proof and the
fact that
$$
1/v_s \le \omega +1/v_s \in \mascP _E(\rr {2d}).
$$

\par

In order to prove the last equality in \eqref{ModCompSuppInv}
we again start to consider the case when
$\omega \ge 1/v_s$ for some $s$ satisfying $1<s<t$.
Let $f\in \maclE _t'$, and choose $\phi$ and
$\psi$ here above such that $\phi =1$ on $\supp f$, $\psi =1$
on $\supp \phi$ and such that $\phi (\cdo -x_j)=0$ on $\supp f$
when $x_j\neq 0$. This is possible in view of Section 3 in \cite{JPTT}.
Also let $Q$ be a closed parallelepiped such that
$$
\bigcup _{k\in J}(\xi _k+Q) =\rr d
$$
and that the intersection of two different $\xi _k+Q$ is a zero set.

\par

Then there is a constant $C>0$ such that
\begin{equation}\label{ModfEquiv}
C^{-1}\nm {f}{M^{p,q}_{(\omega )}}
\le
\nm {e^{i\scal \cdo {\eta }}f}{M^{p,q}_{(\omega )}}
\le
C\nm {f}{M^{p,q}_{(\omega )}},\qquad \eta \in Q.
\end{equation}
Furthermore, by the support properties of $\phi$ and $f$, and
using the fact that the Gabor coefficients $c_\eta (j,k)$ of 
$e^{i\scal \cdo {\eta }}f$ are given by
$$
c_\eta (j,k) = (V_\phi f)(x_j,\xi _k-\eta ),
$$
are zero when $x_j\neq 0$, and 
\begin{equation}\label{cCase0}
c_\eta (0,k) = (V_\phi f)(\xi _k-\eta ) = \widehat f(\xi _k-\eta ).
\end{equation}
Hence, \eqref{ModfEquiv} gives
$$
\nm {f}{M^{p,q}_{(\omega )}} ^q\asymp
\nm {c_\eta}{\ell ^{p,q}_{(\vartheta )}} ^q =
\nm {c_\eta (0,\cdo )\vartheta (0,\cdo )}{\ell ^q}^q.
$$

\par

By integrating the last relations with respect to $\eta$ over $Q$
it follows from \eqref{cCase0} that
$$
\nm {f}{M^{p,q}_{(\omega )}} ^q\asymp 
\nm {\{ \nm {\widehat f}{L^q_{(\omega)}(\xi _k +Q)} \}
_{k\in J}}{\ell ^q} ^q
=
\nm {\widehat f}{L^q_{(\omega)}(\rr d)}^q,
$$
and last equality in \eqref{ModCompSuppInv} follows in this case.

\par

Next assume that $\omega$ is arbitrary, and let $1<s<t$. By Lemma
\ref{LemmaSTFTCompSupp} we have
$$
\mathscr F L^{q}_{(\omega )} \bigcap \maclE '_t
=
\left ( \mathscr F L^{q}_{(\omega )}\bigcap
\mathscr F L^{q}_{(1/v_s )} \right )
\bigcap \maclE '_t
= 
\mathscr F L^{q}_{(\omega +1/v_s)} \bigcap \maclE '_t .
$$
The last equality in \eqref{ModCompSuppInv} now follows from
these identities, the previous case and \eqref{MCapComp}.
The proof is complete.
\end{proof}

\par

We finish the section by applying the previous result on compactly
supported symbols to pseudo-differential operators. (See
Sections 1 and 4 in \cite{Toft11} for strict definitions.) Let
$t\in \mathbf R$, $p\in (0,\infty ]$ and $\omega _1,\omega _2
\in \mascP _E(\rr {2d})$. Then the set $s_{t,p}(\omega _1,\omega _2)$
consists of all $a\in \Sigma _1'(\rr {2d})$ such that the operator
$\op _t(a)$ from $\Sigma _1(\rr d)$ to $\Sigma _1'(\rr d)$ extends
(uniquely) to a Schatten-von Neumann operator from
$M^{2,2}_{(\omega _1)}(\rr d)$ to $M^{2,2}_{(\omega _2)}(\rr d)$. The
following result follows immediately from Theorem A.3 in \cite{Toft11}
and Proposition \ref{CompSuppInv}.

\par

\begin{prop}\label{CompSchatten}
Let $\omega _1,\omega _2\in \mascP _E(\rr {2d})$ and
$\omega _0\in \mascP _E(\rr {4d})$ be such that
$$
\frac {\omega _2(x,\xi )}{\omega _1(y,\eta )}
\asymp
\omega _0((1-t)x+ty,t\xi +(1-t)\eta ,\xi -\eta ,y-x).
$$
Also let $s>1$, $p\in (0,\infty ]$ and $q\in [1,\infty ]$. Then
\begin{multline*}
s_{t,q}(\omega _1,\omega _2) \bigcap \maclE _s'(\rr {2d})
= M^{p,q}_{(\omega _0)}(\rr {2d})\bigcap \maclE _s'(\rr {2d})
\\[1ex]
= \mathscr FL^{q}_{(\omega _0)}(\rr {2d})\bigcap \maclE _s'(\rr {2d}).
\end{multline*}
\end{prop}

\par

\begin{rem}
Propositions \ref{CompSuppInv} and \ref{CompSchatten} remains
true if $\maclE '_t$ are replaced by compactly supported elements
in $\Sigma _t'$, for $t>1$, or by elements in $\mascE '$.
We leave the modifications to the reader.
\end{rem}

\medspace

\end{document}